\documentclass[12pt,letterpaper]{article}
\usepackage[utf8]{inputenc}
\usepackage[all]{xy}
\usepackage{tikz}
\usepackage{mathtools}
\usepackage{bbm}
\usepackage{tikz-cd}
\usepackage{bm}
\usepackage{mathrsfs}
\usepackage{cite}
\usepackage{amsmath}
\numberwithin{equation}{section}
\usepackage{hyperref}
\usepackage{url}
\usepackage{lineno}
\usepackage{mathrsfs}
\usepackage{amsfonts}
\usepackage{amsthm}
\usepackage{amssymb}
\theoremstyle{definition}
\newtheorem{define}{Definition}[subsection]
\newtheorem{example}[define]{Example}
\newtheorem{construction}[define]{Construction}
\theoremstyle{remark}
\newtheorem{remark}[define]{Remark}
\theoremstyle{plain}
\newtheorem{theo}[define]{Theorem}
\newtheorem{lemma}[define]{Lemma}
\newtheorem{prop}[define]{Proposition}

\newcommand{\X}{\mathscr X}
\newcommand{\C}{\mathscr C}

\newcommand{\F}{\mathscr F}
\newcommand{\D}{\mathscr D}

\newcommand{\cat}{\mathbf{Cat}}

\newcommand{\set}{\mathbf{Set}}

\newcommand{\dfib}{\mathbf{DFib}}

\newcommand{\Span}{\mathbf{Span}}

\newcommand{\dblset}{\mathbb{S}\mathbf{et}}
\newcommand{\mon}{\mathbf{Mon}}
\newcommand{\elt}{\mathbf{Elt}}
\newcommand{\dblelt}{\mathbb E\mathbf{l}}
\newcommand{\dblfib}{\mathbb D\mathbf{Fib}}
\newcommand{\dbllax}{\mathbb L\mathbf{ax}}

\newcommand{\spn}{\mathbb S\mathbf{pan}}
\newcommand{\prof}{\mathbb P\mathbf{rof}}

\newcommand{\dblmod}{\mathbb M\mathbf{od}}
\newcommand{\lax}{\mathbf{Lax}}
\newcommand{\ob}{\mathrm{ob}}
\newcommand{\Ob}{\mathbf{Ob}}
\newcommand{\vdbl}{\mathbf{vDbl}}
\newcommand{\dom}{\mathrm{dom}}
\newcommand{\cod}{\mathrm{cod}}
\newcommand{\src}{\mathrm{src}}
\newcommand{\tgt}{\mathrm{tgt}}

\makeatother
\usepackage{amssymb}
\usepackage[left=2cm,right=2cm,top=2cm,bottom=2cm]{geometry}
\makeatletter
\def\slashedarrowfill@#1#2#3#4#5{%
  $\m@th\thickmuskip0mu\medmuskip\thickmuskip\thinmuskip\thickmuskip
  \relax#5#1\mkern-7mu%
  \cleaders\hbox{$#5\mkern-2mu#2\mkern-2mu$}\hfill
  \mathclap{#3}\mathclap{#2}%
  \cleaders\hbox{$#5\mkern-2mu#2\mkern-2mu$}\hfill
  \mkern-7mu#4$%
}
\def\rightslashedarrowfill@{%
  \slashedarrowfill@\relbar\relbar\mapstochar\rightarrow}
\newcommand\xslashedrightarrow[2][]{%
  \ext@arrow 0055{\rightslashedarrowfill@}{#1}{#2}}
\makeatother
\def\slashedrightarrow{\xslashedrightarrow{}}

\begin{document}

\title{Discrete Double Fibrations}

\author{Michael Lambert}

\maketitle

\begin{abstract}
Presheaves on a small category are well-known to correspond via a category of elements construction to ordinary discrete fibrations over that same small category.  Work of R. Par\' e proposes that presheaves on a small double category are certain lax functors valued in the double category of sets with spans.  This paper isolates the discrete fibration concept corresponding to this presheaf notion and shows that the category of elements construction introduced by Par\' e leads to an equivalence of virtual double categories.
\end{abstract}

\tableofcontents

\section{Introduction}

A discrete fibration is functor $F\colon\F\to\C$ such that for each arrow $f\colon C\to FY$ in $\C$, there is a unique arrow $f^*Y\to Y$ in $\F$ whose image under $F$ is $f$. Such functors correspond to ordinary presheaves
\[ \dfib(\C) \simeq [\C^{op},\set]
\]
via a well-known ``category of elements" construction \cite{MacLane}. This is a ``representation theorem" in the sense that discrete fibrations over $\C$ correspond to set-valued representations $\C$.

Investigation of higher-dimensional structures leads to the question of analogues of such well-established developments in lower-dimensional settings. In the case of fibrations, for example, ordinary fibrations have their analogues in ``2-fibrations" over a fixed base 2-category \cite{Buckley}. Discrete fibrations over a fixed base 2-category have their 2-dimensional version in ``discrete 2-fibrations" \cite{MeFibrations}. The justification that the notion of (discrete) fibration is the correct one in each case consists in the existence of a representation theorem taking the form of an equivalence with certain representations of the base structure via a category of elements construction.

In the double-categorical world, R. Par\'e proposes \cite{YonedaThyDblCatz} that certain span-valued, lax functors on a double category $\mathbb B$ are presheaves on $\mathbb B$. This paper aims to isolate the notion of ``discrete double fibration" corresponding to this notion of presheaf. As justification of the proposed definition, Par\'e's double category of elements will be used to exhibit a representation theorem on the pattern of those reviewed above. The technically interesting part of the results is that ultimately an equivalence of virtual double categories is achieved. For this the language of monoids and modules in virtual double categories as in \cite{Operads} will be used.

\subsection{Overview and Motivation}

To give away the answer completely, a discrete double fibration should be a category object in the category of discrete fibrations. This follows the approach to double category theorizing that a ``double \_\_\_" is a category object in the category of \_\_\_'s. The point of the paper is thus to show that this expected definition comports with Par\'e's notion of presheaf by exhibiting the representation theorem as discussed above. This will be proved directly for all the 1-categorical structure involved. 

However, lax functors and their transformations are fragments of a higher double-dimensional structure. That is, modules and their multimodulations make lax functors into a virtual double category. The question, then, arises as to the corresponding virtual double category structure on the discrete fibration side of the desired equivalence. It turns out that there is already well-established language for these structures from \cite{Operads}.

For a category with finite limits $\C$, the category of monoids in the bicategory of spans in $\C$ is equivalent to the category of category objects in $\C$. In fact the correspondence is much stronger than this, since -- under the assumption of a fixed universe of concrete sets -- the data on either side of the equivalence is the same modulo rearrangement of tuples. This general result leads to the special case for $\C = \dfib$, yielding the equivalence
\[ \mon(\Span(\dfib)) \simeq \cat(\dfib).
\]
This is the main clue leading to a natural candidate for virtual double category structures on discrete double fibrations. For monoids in the bicategory part of any double category $\mathbb D$ form the 1-categorical part of the virtual double category of monoids and modules in $\mathbb D$. That is, there is a virtual double category $\dblmod(\mathbb D)$ of modules whose underlying 1-category is precisely monoids in the bicategory part of $\mathbb D$. This leads to the natural candidate for virtual double category structure on discrete double fibrations as $\dblmod(\spn(\dfib))$. The ultimate objective of the paper is thus to show that an appropriate slice of such modules is equivalent to the virtual double category of lax presheaves via the double category of elements construction.

\subsection{Organization and Results}

Section 2 reviews lax functors, their transformations and the elements construction from Par\'e's \cite{YonedaThyDblCatz}. The main definition of ``discrete double fibration" is given as a double functor $P\colon\mathbb E\to\mathbb B$ for which both $P_0$ and $P_1$ are discrete fibrations. This is equivalently a category object in the category of discrete fibrations. The section culminates in the first main contribution of the paper, proved as Theorem \ref{Main Theorem 1} below. This result says that for any strict double category $\mathbb B$, there is an equivalence of 1-categories
\[ \dfib(\mathbb B) \simeq \lax(\mathbb B^{op},\spn)
\]
between the category of discrete double fibrations over $\mathbb B$ and span-valued lax functors on $\mathbb B$, induced by the double category of elements construction. This is achieved by constructing a pseudo-inverse for the double category of elements.

Section 3 introduces the virtual double category structure on lax span-valued functors as in \cite{YonedaThyDblCatz} and \cite{ModuleComposition} with the goal of extending the theorem above to an equivalence of virtual double categories. It turns out that a convenient language for setting up the virtual double category structure on discrete double fibrations is that of monoids and modules in virtual double categories as in \cite{Operads} and later \cite{FrameworkGenMulticatz}.

Section 4 extends the elements construction and its pseudo-inverse to functors of virtual double categories and culminates in the second result, proved as Theorem \ref{Main Theorem 2}. That is, for any strict double category $\mathbb B$, there is an equivalence of virtual double categories
\[ \dblfib(\mathbf B) \simeq \dbllax(\mathbb B^{op},\spn)
\]
induced by the elements functor.

\subsection{Conventions and Notation}

Double categories go back to Ehresmann \cite{Ehresmann}.  Other references include \cite{GP}, \cite{GP2}, and \cite{FramedBicats}. By  the phrase ``double category" we will always mean a strict double category, which is a category object in categories. Adaptations for the pseudo-case in some cases are easy to make but tedious; in other cases these are more subtle and require separate treatment. 

One quirk of our presentation is that, owing to different conventions concerning which arrows are ``vertical" and which are ``horizontal," we prefer to use the more descriptive language of ``arrows," ``proarrows," and ``cells."  Arrows are of course the ordinary arrows between objects (the horizontal arrows in \cite{GP} and \cite{GP2} vs. the vertical ones in \cite{FramedBicats} and \cite{FrameworkGenMulticatz}); whereas the proarrows are the objects of the bicategory part (that is, the vertical ones of \cite{GP} and \cite{GP2} vs. the horizontal ones of \cite{FramedBicats} and \cite{FrameworkGenMulticatz}).  In this language, the arrows of the canonical example of categories and profunctors would be the ordinary functors while the proarrows are the profunctors.  This choice is inspired by the language of a ``proarrow equipment" in \cite{Pro1} and \cite{Pro2}, which sought to axiomatize this very situation as a 2-category ``equipped with proarrows."

Throughout use blackboard letters $\mathbb A$, $\mathbb B$, $\mathbb C$, $\mathbb D$ for double categories.  The 0-part of such $\mathbb D$ is the category of objects and arrows, denoted by `$\mathbb D_0$'.  The 1-part is the category of proarrows and cells, denoted by `$\mathbb D_1$'.  External composition is written with a tensor `$\otimes$'.  The external unit is denoted with $u\colon \mathbb D_0\to\mathbb D_1$. External source and target functors are $\src, \tgt\colon \mathbb D_1\rightrightarrows \mathbb D_0$.  Internal structure is denoted by juxtaposition in the usual order. The terms ``domain" and ``codomain" refer to internal structure, whereas ``source" and ``target" refer to external structure. One result requires the notion of the ``transpose" of a double category $\mathbb D$, denoted here by $\mathbb D^\dagger$. When $\mathbb D$ is strict, $\mathbb D^\dagger$ is a strict double category and transposing extends to a functor $(-)^\dagger\colon \mathbf{Dbl} \to \mathbf{Dbl}$ on the 1-category of strict double categories and double functors.

Script letters $\C$, $\D$, $\X$, denote ordinary categories with objects sets $\C_0$ and arrow sets $\C_1$.  Well-known 1-, 2-, and bi-categories referred to throughout are given in boldface such as $\set$, $\cat$, $\mathbf{Span}$, and $\mathbf{Rel}$. We always refer to the highest level of structure commonly associated with the objects of the category.  Thus, $\cat$ is the 2-category of categories, unless specified otherwise. The main exception is that $\dfib$ is the ordinary category of discrete fibrations. Throughout we take for granted the notions of monoid and category internal to a given category with finite limits. Some 2-categorical concepts are referenced but not in a crucial way. Common double-categories are presented in mixed typeface and named by their proarrows.  In particular $\prof$ is the double category of profunctors; $\mathbb R\mathbf{el}$ is the double category of sets and relations; $\spn$ (rather than $\dblset$) is the double category of sets and spans.

\section{Lax Functors, Elements, and Discrete Double Fibrations}

The main objects of study are lax functors between double categories. The definitions are recalled in this section. The double category of elements is recalled, and it is seen how its construction leads naturally to the definition of a discrete double fibration.

\subsection{Lax Functors}

As groundwork, recall here the definitions of lax functor and their natural transformations.

\begin{define}[See e.g. \S 7.2 of \cite{GP} or \S 1.1.2 of \cite{ModuleComposition}] \label{lax functor defn} A \textbf{lax functor} between (pseudo-) double categories $F\colon\mathbb A\to\mathbb B$ consists of assignments
\[ A\mapsto FA\qquad f\mapsto Ff\qquad m\mapsto Fm\qquad \alpha\mapsto F\alpha
\]
on objects, arrows, proarrows and cells, respecting domains, codomains, sources and targets; and laxity cells
$$\xymatrix{
\ar@{} [dr]|{\phi_A} FA \ar@{=}[d] \ar[r]^{u_{FA}}|-@{|} & FA \ar@{=}[d] & & \ar@{} [drr]|{\phi_{n,m}} FA \ar@{=}[d] \ar[r]^{Fm}|-@{|} & FB \ar[r]^{Fn}|-@{|} & FC \ar@{=}[d] \\
FA \ar[r]_{Fu_A}|-@{|} & FA & & FA \ar[rr]_{F(n\otimes m)}|-@{|} &    & FC
}$$
for each object $A$ and composable proarrows $m\colon A\slashedrightarrow B$ and $n\colon B\slashedrightarrow C$, all subject to the following axioms.
\begin{enumerate}
\item{}[Internal Functoriality] The object and arrows assignments $A\mapsto FA$, $f\mapsto Ff$ above define an ordinary functor $F_0\colon \mathbb A_0\to\mathbb B_0$; and the proarrow and cell assignments define an ordinary functor $F_1\colon \mathbb A_1\to\mathbb B_1$ with respect to internal identities and composition of cells.
\item{}[Naturality] For any $f\colon A\to B$, there is an equality
$$\xymatrix{ \ar@{} [dr]|{u_{Ff}} \cdot \ar[d]_{Ff} \ar[r]^{u_{FA}}|-@{|} & \cdot \ar[d]^{Ff} & & \ar@{} [dr]|{\phi_A} \cdot \ar@{=}[d] \ar[r]^{u_{FA}}|-@{|} & \cdot \ar@{=}[d] \\
\ar@{} [dr]|{\phi_B} \cdot \ar@{=}[d] \ar[r]|-@{|} & \cdot \ar@{=}[d] & = & \ar@{}[dr]|{Fu_f} \cdot \ar[d]_{Ff} \ar[r]|-@{|} & \cdot \ar[d]^{Ff} \\
\cdot \ar[r]_{Fu_B}|-@{|} & \cdot & & \cdot \ar[r]_{Fu_B}|-@{|} & \cdot
}$$
and for any externally composable cells $\alpha\colon m\Rightarrow p$ and $\beta\colon n\Rightarrow q$, there is an equality
$$\xymatrix{ \ar@{}[dr]|{F\alpha} \cdot \ar[d]_{Ff} \ar[r]^{Fm}|-@{|} & \ar@{}[dr]|{F\beta} \cdot \ar[d]^{} \ar[r]^{Fn}|-@{|} & \cdot \ar[d]^{Fh} && \ar@{}[drr]|{\phi_{n,m}} \cdot \ar@{=}[d] \ar[r]^{Fm}|-@{|} & \cdot \ar[r]^{Fn}|-@{|} & \cdot \ar@{=}[d] \\
\ar@{}[drr]|{\phi_{q,p}} \cdot \ar@{=}[d] \ar[r]_{Fp}|-@{|} & \cdot \ar[r]_{Fq}|-@{|} & \cdot \ar@{=}[d] & = & \ar@{}[drr]|{F(\beta\otimes\alpha)} \cdot \ar[d]_{Ff} \ar[rr]|{F(n\otimes m)} & & \cdot \ar[d]^{Ff} \\
\cdot \ar[rr]_{F(q\otimes p)}|-@{|} & & \cdot & & \cdot \ar[rr]_{F(q\otimes p)}|-@{|} & & \cdot 
}$$
of composite cells.
\item{}[Unit and Associativity] Given a proarrow $m\colon A\slashedrightarrow B$, the unit laxity cells satisfy
$$\xymatrix{ \ar@{}[dr]|{\phi_A} \cdot \ar@{=}[d] \ar[r]^{u_{FA}}|-@{|} & \ar@{}[dr]|{Fu_m} \cdot \ar[d]^{} \ar[r]^{Fm}|-@{|} & \cdot \ar@{=}[d] && \ar@{}[drr]|{\cong} \cdot \ar@{=}[d] \ar[r]^{u_{FA}}|-@{|} & \cdot \ar[r]^{Fm}|-@{|} & \cdot \ar@{=}[d] \\
\ar@{}[drr]|{\phi_{m, u_A}} \cdot \ar@{=}[d] \ar[r]_{Fu_A}|-@{|} & \cdot \ar[r]_{Fm}|-@{|} & \cdot \ar@{=}[d] & = & \ar@{}[drr]|{\cong} \cdot \ar@{=}[d] \ar[rr]|{Fm} & & \cdot \ar@{=}[d] \\
\cdot \ar[rr]_{F(m\otimes u_A)}|-@{|} & & \cdot & & \cdot \ar[rr]_{F(m\otimes u_A)}|-@{|} & & \cdot 
}$$
and 
$$\xymatrix{ \ar@{}[dr]|{Fu_m} \cdot \ar@{=}[d] \ar[r]^{Fm}|-@{|} & \ar@{}[dr]|{\phi_B} \cdot \ar[d]^{} \ar[r]^{u_{FB}}|-@{|} & \cdot \ar@{=}[d] && \ar@{}[drr]|{\cong} \cdot \ar@{=}[d] \ar[r]^{Fm}|-@{|} & \cdot \ar[r]^{u_{FB}}|-@{|} & \cdot \ar@{=}[d] \\
\ar@{}[drr]|{\phi_{u_B, m}} \cdot \ar@{=}[d] \ar[r]_{Fm}|-@{|} & \cdot \ar[r]_{Fu_B}|-@{|} & \cdot \ar@{=}[d] & = & \ar@{}[drr]|{\cong} \cdot \ar@{=}[d] \ar[rr]|{Fm} & & \cdot \ar@{=}[d] \\
\cdot \ar[rr]_{F(u_B\otimes m)}|-@{|} & & \cdot & & \cdot \ar[rr]_{F(u_B\otimes m)}|-@{|} & & \cdot 
}$$
and for any three composable proarrows $m\colon A\slashedrightarrow B$, $n\colon B\slashedrightarrow C$, and $p\colon C\slashedrightarrow D$, the laxity cells are associative in the sense that
$$\xymatrix{ \ar@{}[dr]|{Fu_m}\cdot \ar@{=}[d] \ar[r]^{Fm}|-@{|} & \ar@{}[drr]|{\phi_{p,n}} \cdot \ar@{=}[d] \ar[r]^{Fn}|-@{|} & \cdot \ar[r]^{Fp}|-@{|} & \cdot \ar@{=}[d] & & & & \ar@{}[drr]|{\phi_{n,m}}\cdot \ar@{=}[d] \ar[r]^{Fm}|-@{|} & \cdot \ar[r]^{Fn}|-@{|} & \cdot \ar@{}[dr]|{Fu_p} \ar@{=}[d]  \ar[r]^{Fp}|-@{|} & \cdot \ar@{=}[d] \\
\ar@{}[drrr]|{\phi_{p\otimes n, m}} \cdot \ar@{=}[d] \ar[r]|-@{|} & \cdot \ar[rr]|-@{|} &   & \cdot \ar@{=}[d] & & = & & \ar@{}[drrr]|{\phi_{p,n\otimes m}} \cdot \ar@{=}[d] \ar[rr]|-@{|} &   & \cdot \ar[r]|-@{|} & \cdot \ar@{=}[d]\\
\cdot \ar[rrr]_{F((p\otimes n)\otimes m)}|-@{|} &  &  & \cdot  & & & & \cdot \ar[rrr]_{F(p\otimes (n\otimes m))}|-@{|}&  &  & \cdot\\
}$$
are equal, modulo composition with the image under $F$ of the associativity iso cell 
\[(p\otimes n)\otimes m\cong p\otimes (n\otimes m)
\]
given with the structure of $\mathbb A$. 
\end{enumerate}
Notice that the structural isomorphism reduce to strict identities when $\mathbb A$ and $\mathbb B$ are strict double categories. A \textbf{pseudo double functor} is a lax double functor where the laxity cells are isomorphisms; a \textbf{strict double functor} is one whose laxity cells are identities.
\end{define}

\begin{remark}  Generally speaking, use lower-case Greek letters for the laxity cells corresponding to the Latin capital letter for the functor.  Thus, $F$ is a lax functor with laxity cells denoted by `$\phi$' with subscripts; `$\gamma$' is used for a lax functor $G$.
\end{remark}

\begin{example}[Cf. \S 1.2 of \cite{YonedaThyDblCatz}]   The usual object functor $\ob(-)\colon \cat\to\set$ extends to a lax double functor $\Ob(-)\colon \prof\to \spn$ in the following way.  On a category $\C$ take the object set $\C_0$ to be the image as usual; likewise take the object part $F_0\colon \C_0\to\D_0$ to be the image of a given functor $F\colon \C\to\D$.  On a profunctor $P\colon \C\slashedrightarrow \D$ -- that is, a functor $P\colon \C^{op}\times\D\to\set$ -- take the image to be the disjoint union
\[ \Ob(P):=\coprod_{C,D}P(C,D).
\]
The assignment on a given transformation of profunctors is induced by the universal property of the coproduct.  This is a bona fide lax functor
\end{example}

\begin{define}[Cf. \S 1.1.5 \cite{ModuleComposition}] \label{define: lax nat transf} Let $F, G\colon\mathbb A\rightrightarrows\mathbb B$ denote lax functors with laxity cells $\phi$ and $\gamma$.  A \textbf{natural transformation} $\tau\colon F \to G$ assigns to each object $A$ an arrow $\tau_A\colon FA\to GA$ and to each proarrow $m\colon A\slashedrightarrow B$ a cell
$$\xymatrix{ \ar @{} [dr] |{\tau_m} FA \ar[r]^{Fm}|-@{|} \ar[d]_{\tau_A} & FB \ar[d]^{\tau_B}  \\  GA \ar[r]_{Gm}|-@{|} & GB
}$$
in such a way that the following axioms are satisfied.
\begin{enumerate}
\item{}[Naturality] Given any arrow $f\colon A\to B$, the usual naturality square
$$\xymatrix{  FA \ar[r]^{Ff} \ar[d]_{\tau_A} & FB \ar[d]^{\tau_B}  \\  GA \ar[r]_{Gf} & GB
}$$
commutes; and given any cell $\alpha\colon m\Rightarrow n$, the corresponding naturality square commutes in the sense that the compositions 
$$\xymatrix{ \ar@{} [dr]|{F\alpha} \cdot \ar[d]_{} \ar[r]^{Fm}|-@{|} & \cdot \ar[d]^{} & & \ar@{} [dr]|{\tau_m} \cdot \ar@{=}[d] \ar[r]^{Fm}|-@{|} & \cdot \ar@{=}[d] \\
\ar@{} [dr]|{\tau_n} \cdot \ar@{=}[d] \ar[r]|-@{|} & \cdot \ar@{=}[d] & = & \ar@{}[dr]|{G\alpha} \cdot \ar[d]_{} \ar[r]|-@{|} & \cdot \ar[d]^{} \\
\cdot \ar[r]_{Gn}|-@{|} & \cdot & & \cdot \ar[r]_{Gn}|-@{|} & \cdot
}$$
are equal.
\item{}[Functoriality] For any object $A$, the compositions
$$\xymatrix{ \ar@{} [dr]|{u_{\tau_A}} \cdot \ar[d]_{} \ar[r]^{u_{FA}}|-@{|} & \cdot \ar[d]^{} & & \ar@{} [dr]|{\phi_A} \cdot \ar@{=}[d] \ar[r]^{u_{FA}}|-@{|} & \cdot \ar@{=}[d] \\
\ar@{} [dr]|{\gamma_A} \cdot \ar@{=}[d] \ar[r]|-@{|} & \cdot \ar@{=}[d] & = & \ar@{}[dr]|{\tau_{u_A}} \cdot \ar[d]_{} \ar[r]|-@{|} & \cdot \ar[d]^{} \\
\cdot \ar[r]_{Gu_A}|-@{|} & \cdot & & \cdot \ar[r]_{Gu_A}|-@{|} & \cdot
}$$
are equal; and for any composable proarrows $m\colon A\slashedrightarrow B$ and $n\colon B\slashedrightarrow C$, the composites 
$$\xymatrix{ \cdot\ar@{}[dr]|{\tau_m}\ar[d]_{\tau_A}\ar[r]^{Fm}|-@{|} &\cdot\ar@{}[dr]|{\tau_n}\ar[d] \ar[r]^{Fn}|-@{|} &\cdot\ar[d]^{\tau_C} && \cdot\ar@{}[drr]|{\phi_{n,m}} \ar@{=}[d] \ar[r]^{Fm}|-@{|} & \cdot \ar[r]^{Fn}|-@{|} & \cdot \ar@{=}[d] \\
\cdot\ar@{}[drr]|{\gamma_{n,m}}\ar@{=}[d]\ar[r]_{Gm}|-@{|}&\cdot\ar[r]_{Gn}|-@{|}&\cdot\ar@{=}[d]&=&\cdot\ar@{}[drr]|{\tau_{n\otimes m}}\ar[d]_{\tau_A} \ar[rr]|-@{|} && \cdot\ar[d]^{\tau_C}\\
\cdot\ar[rr]_{G(n\otimes m)}|-@{|} &&\cdot  &&\cdot\ar[rr]_{G(n\otimes m)}|-@{|}&&\cdot
}$$
are equal.
\end{enumerate}
Let $\lax(\mathbb A,\mathbb B)$ denote the category of lax functors $\mathbb A\to\mathbb B$ and their transformations. The composition is ``component-wise" and identity transformations are those with identity morphisms in each of their components.
\end{define}

In a few simple cases, this category can be described in terms of well-known structures. These equivalences are given by the ``elements" construction at the beginning of the next subsection.

\begin{example}  A lax functor $\mathbf 1\to\spn$ is essentially a small category.  In fact, taking elements induces an equivalence
\[ \mathbf{Lax}(\mathbf 1,\spn)\simeq \cat.
\]
In terms of monoids and their morphisms (i.e. a 1-category of monoids and their morphisms in a bicategory), this means that
\[  \mathbf{Lax}(\mathbf 1,\spn)\simeq \mathbf{Mon}(\Span(\set)).
\]
as 1-categories. This is a primordial example or special case of our main results, Theorem \ref{Main Theorem 1} and Theorem \ref{Main Theorem 2}.
\end{example}

\begin{example}[See \S 3.13 of \cite{YonedaThyDblCatz}]  If $\C$ is an ordinary category, let $\mathcal V\C$ denote the ``vertical double category" formed from the objects of $\C$ and using the morphisms of $\C$ as the proarrows with only identity arrows and cells.  There is then an equivalence
\[ \mathbf{Lax}(\mathcal V\C,\spn)\simeq \cat/\C
\]
of ordinary categories given by taking elements.
\end{example}

\subsection{Discrete Double Fibrations}

The fibration properties of the double category of elements construction lead to the axiomatization of our notion of ``discrete double fibration." Let us recall the details of this double category and its associated projection functor.

\begin{construction}[\S 3.7 \cite{YonedaThyDblCatz}] \label{double category of elements}  Let $F\colon \mathbb B^{op}\to\spn$ denote a lax functor with laxity cells $\phi$.  The \textbf{double category of elements} $\dblelt(F)$ has 
\begin{enumerate} 
\item as objects those pairs $(B,x)$ with $x\in FB$; 
\item as morphisms $f\colon (B,x)\to (C,y)$ those arrows $f\colon B\to C$ of $\mathbb B$ such that $f^*y=x$ holds under the action of the transition function $f^*\colon FC\to FB$;
\item as proarrows $(B,x)\slashedrightarrow (C,y)$ those pairs $(m,s)$ consisting of a proarrow $m\colon B\slashedrightarrow C$ and an element $s\in Fm$ such that $s_0=x$ and $s_1=y$ both hold; and 
\item as cells $(m,s)\Rightarrow (n,t)$ those cells $\theta$ of $\mathbb B$ as at left below for which the associated morphism of spans on the right
$$\xymatrix{
\ar@{}[dr]|{\theta}A \ar[d]_{f} \ar[r]^{m}|-@{|} & B \ar[d]^g & & FC \ar[d]_{f^*} & \ar[l]_{(-)_0} Fn \ar[d]^{\theta^*} \ar[r]^{(-)_1} & FD \ar[d]^{g^*} \\
C \ar[r]_n|-@{|} & D & & FA & \ar[l]^{(-)_0} Fm \ar[r]_{(-)_1} & FB
}$$
satisfies $\theta^*(t)=s$.  
\end{enumerate}
The 1-category structure on objects and morphisms is the same as the ordinary category of elements.  Internal composition of cells uses the internal composition in $\mathbb B$ and the strict equalities of the form $\theta^*(t)=s$ in the definition.  External composition uses the given laxity morphisms.  For example, given composable proarrows $(m,s)\colon (B,x)\slashedrightarrow (C,y)$ and $(n,t)\colon (C,y)\slashedrightarrow (D,z)$, the composite is defined as
\[ (n,t) \otimes (m,s):= (n\otimes m,\phi_{n,p}(s,t)).
\]
Of course this is a well-defined proarrow.  Identities and external composition of cells is similar.  This makes $\dblelt(F)$ a double category.  It is only as strict as $\mathbb B$.  There is a \textbf{projection double functor} $\Pi\colon \dblelt(F)\to\mathbb B$ taking the indexing objects, morphisms, proarrows and cells to $\mathbb B$.  This is strict even if $\mathbb B$ is pseudo.  \end{construction}

It is remarked in \cite{YonedaThyDblCatz} that taking elements extends to a functor. Here are the details.

\begin{construction}[Morphism from a Transformation of Lax Functors] Let $\tau\colon F\to G$ denote a transformation of lax functors $F,G\colon\mathbb B^{op}\rightrightarrows \spn$ as in Definition \ref{define: lax nat transf}. Define what will be a strict functor of double categories $\dblelt(\tau)\colon\dblelt(F)\to\dblelt(G)$ over $\mathbb B$. On objects and arrows, take
\[ (D,x) \mapsto (D, \tau_Dx) \qquad\qquad f\mapsto f
\]
The arrow assignment is well-defined because $f^*\tau_D = \tau_Cf^*$ holds by the strict naturality condition in the definition. This assignment is functorial by construction. Now, for proarrows, send
\[ (m,s)\mapsto (m,\tau_ms)\qquad\qquad\alpha\mapsto \alpha
\]
similarly to the object and arrow assignments. Well-definition again follows from naturality. Functoriality is by construction. The rest of the important content is in the following result.
\end{construction}

\begin{lemma} \label{lemma: elements is a functor} Given a transformation $\tau\colon F\to G$ of span-valued, lax functors, the assignments above yield a morphism from $\Pi\colon\dblelt(F)\to\mathbb B$ to $\Pi\colon\dblelt(G)\to\mathbb B$ over $\mathbb B$. This association is functorial, meaning that the elements functor
\[ \dblelt(-)\colon \lax(\mathbb B^{op},\spn) \longrightarrow \mathbf{Dbl}/\mathbb B
\]
valued in the slice of double categories over $\mathbb B$ is well-defined.
\end{lemma}
\begin{proof} It is left to see that $\dblelt(\tau)\colon \dblelt(F)\to\dblelt(G)$ preserves external composites and units. The definition of external composition in the elements construction incorporates the laxity morphisms coming with $F$ and $G$. So, the preservation of external composition reduces to the statement that these laxity morphisms interact in the proper way with the components of $\tau$ used in the definition. But this is precisely what the ``Functoriality" condition of Definition \ref{define: lax nat transf} axiomatizes. Unit preservation follows similarly.  \end{proof}

As above, the elements functor is valued in the slice of the category of double categories over $\mathbb B$. However, those double functors strictly in its image possess certain fibration properties leading to the definition of a ``discrete double fibration." First the properties:

\lemma  Let $F\colon \mathbb B^{op}\to\spn$ denote a lax double functor.  The projection functors $\Pi_0$ and $\Pi_1$ underlying the canonical projection $\Pi\colon \mathbb E\mathbf{lt}(F)\to \mathbb B$ are discrete fibrations.
\endlemma
\proof Here are the required lifts. Given $(D,x)$ and $f\colon C\to D$, the lift is
\[ f\colon (C,f^*x)\to (D,x)
\]
where $f^*\colon FD\to FC$ is the transition function. Similarly, given $(n,s)$ for a proarrow $n$ and $s\in Fn$ and a cell $\alpha\colon m\Rightarrow n$, the cartesian cell above $\alpha$ with codomain $(n,s)$ is
$$\xymatrix{
(A, f^*x)\ar@{}[dr]|{\alpha} \ar[d]_{f} \ar[r]^{(m,\alpha^*s)}|-@{|}  & (B,g^*y) \ar[d]^{g}\\
(C, x) \ar[r]_{(n, s)}|-@{|} & (D,y)
}$$
where $\alpha^*\colon Fn\to Fm$ is the transition function between the vertices of the spans. Notice that by the fact that $\alpha^*$ induces a morphism of spans, the source and target of $(m, \alpha^*s)$ are well-defined.  \endproof

Such functors are equivalently characterized with a distinctly double categorical flavor.

\begin{prop} \label{characterize disc doubl fibs} A double functor $P\colon \mathbb E\to\mathbb B$ between strict double categories for which $P_0$ and $P_1$ are discrete fibrations is equivalently a category object in $\dfib$, and thus equivalently a monoid in the bicategory of spans in $\dfib$.
\end{prop}
\begin{proof} In the first place $\dfib$ has strict pullbacks, so the statement itself makes sense. On the one hand, a category object in any subcategory of an arrow category closed under finite limits is an internal functor. Thus, a category in $\dfib$ is a double functor. Its object and arrow parts must be objects of $\dfib$, that is, discrete fibrations. On the other hand, a double functor is a category object in the arrow category of $\cat$. But if $P_0$ and $P_1$ are discrete fibrations, then such $P$ lives in $\cat(\dfib)$. A monoid in spans in categories is equivalently a double category; similarly a monoid in spans in the arrow category of $\cat$ is equivalently a double functor. Thus, the components of such a double functor are discrete fibrations if and only if the double functor is in fact a monoid in spans in $\dfib$.  \end{proof}

\begin{define}  A \textbf{discrete double fibration} is a category object in $\dfib$. A morphism of discrete double fibrations is a pair of double functors $(H,K)$ making a commutative square. Take $\cat(\dfib)$ to be the category of discrete double fibrations. Let $\dfib(\mathbb B)$ denote the subcategory of discrete double fibrations with fixed target $\mathbb B$ and morphisms with $K = 1_\mathbb B$. This is equivalently the fiber of the codomain projection $\cod\colon \cat(\dfib) \to \mathbf{Dbl}$ over $\mathbb B$.
\end{define}

\begin{remark} Although the official definition has a succinct and philosophically appropriate phrasing, as it is convenient, any of the equivalent characterizations of discrete double fibrations in Proposition \ref{characterize disc doubl fibs} may be used throughout.
\end{remark}

\begin{remark}  The name used is ``\emph{discrete} double fibration" because this is a discretization of a more general concept of ``double fibration." This will be a double functor $P\colon \mathbb E\to\mathbb B$ whose underlying functors $P_0$ and $P_1$ are fibrations but that satisfy some further compatibility conditions.
\end{remark}

\begin{remark} The results so far means that by fiat the elements functor is valued in discrete double fibrations
\[ \dblelt(-)\colon \lax(\mathbb B^{op},\spn) \longrightarrow \dfib(\mathbb B)
\]
over $\mathbb B$. The purpose of this section is now to exhibit the pseudo-inverse yielding an equivalence of categories. First end this subsection with some general results.
\end{remark}

\begin{example}[Domain Projection]  For any double category, the domain projection functor
\[ \dom\colon \mathbb B/X\longrightarrow \mathbb B
\]   
from the double slice $\mathbb B/X$ is a discrete double fibration.  This is the image under the elements functor $\dblelt(-)\colon \mathbf{Lax}(\mathbb B^{op},\spn)\to \dfib(\mathbb B)$ of the canonical representable functor on $X$.
\end{example}

\begin{prop} \label{pullback characterization} A double functor $P\colon \mathbb E\to\mathbb B$ is a discrete double fibration over a strict double category $\mathbb B$ if, and only if, the transpose square
$$\xymatrix{
\mathbb E^\dagger_1 \ar[d]_{P_1^\dagger}\ar[r]^{\cod} &
\mathbb E^\dagger_0 \ar[d]^{P_0^\dagger} \\
\mathbb B^\dagger_1 \ar[r]_{\cod}&
\mathbb B^\dagger_0
}$$
is a pullback in $\cat$.
\end{prop}
\begin{proof}  Straightforward verification.    \end{proof}

\begin{remark}  Lemma \ref{pullback characterization} is the analogue of the characterization of ordinary discrete fibrations, saying that a functor $F\colon \F\to\C$ is a discrete fibration if, and only if, the square 
$$\xymatrix{
\F_1 \ar[d]_{F_1} \ar[r]^{d_1} & \F_0 \ar[d]^{F_0} \\
\C_1 \ar[r]_{d_1} & \C_0
}$$
is a pullback in $\set$.  This characterization will be important in the monadicity developments in forthcoming work.
\end{remark}

\subsection{Pseudo-Inverse on Objects and Horizontal Morphisms}

The present goal is to construct a pseudo-inverse for the elements functor. This will be a functor from discrete double fibrations back to lax span-valued double functors. For a discrete double fibration $P\colon \mathbb E\to\mathbb B$, begin correspondences leading to a lax double functor $F_P\colon \mathbb B^{op}\to\spn$ in the following way.

\begin{construction}[Transition Morphisms] On objects, take $B\mapsto \mathbb E_B$, the set of objects of $\mathbb E_0$ over $B\in\mathbb B_0$ via $P_0$. Call this the ``fiber" over $B$. Since $P_0$ is a discrete fibration, for every horizontal morphism $f\colon B\to C$, there is a corresponding transition function $f^*\colon \mathbb E_C\to\mathbb E_B$ given on $x\in \mathbb E_C$ by taking the domain $f^*x$ of the unique horizontal morphism over $f$ with codomain $x$. Given a proarrow $m\colon A\slashedrightarrow B$, required is a span from $\mathbb E_A$ to $\mathbb E_B$.  For this, let $\mathbb E_m$ denote the fiber of $P_1$ over $m$, and send $m$ to the span
\[ \mathbb E_A \xleftarrow{\src} \mathbb E_m \xrightarrow{\tgt} \mathbb E_B.
\]  
Finally, associated to a cell $\alpha$ of $\mathbb B$ of the form
$$\xymatrix{
A \ar@{}[dr]|{\alpha} \ar[d]_{f} \ar[r]^{m}|-@{|} & B \ar[d]^{g} \\
C \ar[r]_{n}|-@{|} & D
}$$
will be a cell of $\spn$, that is, a morphism of the spans associated to $m$ and $n$.  In light of the definitions so far, required is a function $\mathbb E_n\to \mathbb E_m$ making the diagram
$$\xymatrix{ 
\mathbb E_C \ar[d] & \ar[l] \mathbb E_n \ar@{-->}[d]^{\alpha^*} \ar[r] & \mathbb E_D \ar[d] \\
\mathbb E_A & \ar[l] \mathbb E_m \ar[r] & \mathbb E_B
}$$
commute.  But such a function $\mathbb E_n\to \mathbb E_m$ is given by the fact that $P_1$ is a discrete fibration.  That is, given $u\in \mathbb E_n$, there is a unique cell of $\mathbb E$ with codomain $u$ over $\alpha$, which, by the fact that $P_0$ is a discrete fibration, must be of the form
$$\xymatrix{
f^*x \ar@{}[dr]|{\Downarrow} \ar[d] \ar[r]|-@{|} & g^*y \ar[d] \\
x \ar[r]|-@{|}_u & y
}$$
with source and target the unique lifts over $f$ and $g$ respectively. Thus, send $u$ to the vertical source of this lifted cell, which, by construction, is over $m$, hence well-defined.  The cell diagram in $\spn$ displayed above then commutes by construction of $\mathbb E_n\to \mathbb E_m$. \end{construction}

\begin{construction}[Laxity Cells I] Let $D$ denote an object of $\mathbb B$. The laxity cell for the corresponding unit proarrow is of the form
$$\xymatrix{
\mathbb E_D \ar@{=}[d] & \ar[l]_{1} \mathbb E_D \ar@{-->}[d] \ar[r]^{1} & \mathbb E_D \ar@{=}[d] \\
\mathbb E_D & \ar[l]^{\src} \mathbb E_{u_D} \ar[r]_{\tgt} & \mathbb E_D
}$$
The functor between vertices $\lambda_D\colon\mathbb E_D \to \mathbb E_{u_D}$ is given by $X \mapsto u_X$. Notice that the unit condition of Definition \ref{lax functor defn} is satisfied because the external composition for $\mathbb E$ is strict.
\end{construction}

\begin{construction}[Laxity Cells II] Let $m\colon A\slashedrightarrow B$ and $n\colon B\slashedrightarrow C$ denote proarrows of $\mathbb B$. Required for $F_P$ are laxity cells for such $m$ and $n$ of the form 
$$\xymatrix{ 
\ar @{} [drrrr] |{\phi_{n,m}} \mathbb E_A  \ar@{=}[d] & \ar[l]_{\src} \mathbb E_m\ar[r]^{\tgt} & \mathbb E_B  & \mathbb E_n \ar[l]_{\src} \ar[r]^{\tgt} & \mathbb E_C \ar@{=}[d]  \\  
\mathbb E_A  & & \ar[ll]^{\src} \mathbb E_{n\otimes m}\ar[rr]_{\tgt} & & \mathbb E_C
}$$
between spans. The composed span in the domain of the cell has vertex given by the pullback. The cell amounts to a functor between vertices respecting the source and target functors. This is given by external composition of proarrows and cells over $m$ and $n$ respectively:
$$\xymatrix{ \mathbb E_m\times_{\mathbb E_B}\ar[d]_{\phi_{n,m}}\mathbb E_n & &   u\colon x\slashedrightarrow y, v\colon y\slashedrightarrow z \ar@{|->}[d] \\
  \mathbb E_{n\otimes m} & & v\otimes u
}$$
This is strictly functorial and a well-defined span morphism by the assumptions on $P\colon\mathbb E\to\mathbb B$. The laxity coherence law follows from the fact that external composition for $\mathbb E$ is associative.
\end{construction}

\lemma \label{pseudo-inverse} Let $P\colon\mathbb E\to\mathbb B$ denote a double fibration. The assignments
\[ D\mapsto \mathbb E_D \qquad f\mapsto f^*\qquad m\mapsto \mathbb E_m\qquad \alpha\mapsto \alpha^*
\]
with unit and laxity cells as above define a lax functor $F_P\colon \mathbb B^{op}\to\spn$.
\endlemma
\proof What remains to check is the naturality conditions in Definition \ref{lax functor defn}. Take cells $\alpha\colon m\Rightarrow p$ and $\beta\colon n\Rightarrow q$ of $\mathbb B$. The first condition amounts to the commutativity of the square made by the functors between vertices:
$$\xymatrix{  
\mathbb E_p\times_{\mathbb E_N}\mathbb E_q \ar[r]^{\;\;\;\phi} \ar[d]_{\alpha^*\times\beta^*} & \mathbb E_{q\otimes p} \ar[d]^{(\beta\otimes \alpha)^*}  \\   
\mathbb E_m\times_{\mathbb E_B}\mathbb E_n \ar[r]_{\;\;\;\phi} & \mathbb E_{n\otimes m}
}$$ 
But this follows by uniqueness assumptions. Naturality for the unit cells follows similarly. \endproof

\begin{construction}[Pseudo-Inverse on Morphisms of Discrete Double Fibrations] \label{pseudo-inverse for morphisms} Start with a morphism $H\colon P\to B$ of discrete double fibrations $P\colon \mathbb E\to\mathbb B$ and $Q\colon \mathbb G\to\mathbb B$.  Define correspondences for what will be a transformation $F_H\colon F_P\to F_Q$ between lax double functors $F_P$ and $F_Q$ arising as in Lemma \ref{pseudo-inverse}.  On objects $B$, define $(F_H)_B$ to be the restriction of $H$ to the fibers $H_0\colon \mathbb E_B\to\mathbb G_B$.  This is well defined since $QH=P$ holds.  Similarly, for a proarrow $m\colon B\slashedrightarrow C$, take the component proarrow $(F_H)_v$ in $\spn$ to be the span
$$\xymatrix{
\mathbb E_B \ar[d]_{H_0} & \ar[l] \mathbb E_m \ar[d]^{H_1} \ar[r] & \mathbb E_C \ar[d]^{H_0} \\
\mathbb G_B & \ar[l] \mathbb G_m \ar[r] & \mathbb G_C 
}$$
Again this is well-defined since $QH=P$ holds.
\end{construction}

\begin{prop} \label{lemma:psd inv is a functor} The assignments in Construction \ref{pseudo-inverse for morphisms} make a transformation $F_H\colon F_P\to F_Q$ of lax functors as in Definition \ref{define: lax nat transf}. These assignments result in a functor
\[ F_{(-)}\colon\lax(\mathbb B^{op},\spn)\longrightarrow \dfib(\mathbb B).
\]
\end{prop}
\begin{proof}  The first naturality condition holds because $H$ commutes with $P$ and $Q$ and because $Q_0$ is a discrete fibration; the second naturality condition holds again because $QH=P$ is true and because $Q_1$ is also a discrete fibration.  Proarrow functoriality follows from the fact that $H$ is a strict double functor.  \end{proof}

\subsection{An Equivalence of 1-Categories}

The pseudo-inverse construction of the last subsection in fact induces an equivalence of categories, leading to the first representation theorem.

\begin{construction} \label{equiv iso 1} To define a natural isomorphism $\eta\colon 1\cong F_{\dblelt(-)}$ between functors, needed are component transformations of lax functors indexed by lax $H\colon \mathbb B^{op}\to\spn$. For such a lax functor $H$, the required transformation of lax functors $\eta_H\colon H\to F_{\dblelt(H)}$ is given by 
\[ \eta_{H,A}\colon HA\longrightarrow \dblelt(H)_A\qquad X\mapsto (A,x)
\]
on objects $A\in |\mathbb B_0|$ and by
\[ \eta_{H,m}\colon Hm\longrightarrow \dblelt(F)_m\qquad s\mapsto (m,s)
\]
for a proarrow $m\colon B\slashedrightarrow C$, giving a morphism of spans
$$\xymatrix{
HB \ar[d]_{\eta} & Hm \ar[l] \ar[d]^{\eta} \ar[r] & HC \ar[d]^{\eta} \\
\dblelt(H)_B & \ar[l] \dblelt(H)_m \ar[r] & \dblelt(H)_C.
}$$
Since the maps $\eta_{H,A}$ and $\eta_{H,v}$ just add in indices, they are both bijections of sets.  Moreover, these components result in a transformation of lax functors as in Definition \ref{define: lax nat transf} by construction.  Thus, such $\eta_H$ is the vertical component of a supposed transformation of functors of virtual double categories. Naturality in $H$ is proved in the theorem below.
\end{construction}

\begin{construction} \label{equiv iso 2} On the other hand, a natural isomorphism $\epsilon\colon\elt(F_{(-)})\cong 1$ is given by components $\epsilon_P\colon \dblelt(F_P) \longrightarrow \mathbb E$ where $P\colon\mathbb E\to\mathbb B$ is a discrete double fibration. On objects and arrows take
\[ (D,x)\mapsto x\qquad (C,x)\xrightarrow{f} (D,y) \mapsto f^*x\xrightarrow{!} y
\]
where $f^*x\to y$ is the unique arrow above $f$ with codomain $y$. It is well-defined because $f^*y=x$ holds. These are bijections by uniqueness of these lifts; they are functorial and respect the fibering over $\mathbb B$. On proarrows and cells, take
\[(m,s)\mapsto s \qquad \alpha\;\mapsto\; \alpha^*s\Rightarrow s
\]
where $\alpha^*s\Rightarrow s$ is the unique lift in $\mathbb E$ of $\alpha$ with codomain $s$. Again these are bijections, functorial and fiber-respecting. Naturality in $P$ is proved in the following result. \end{construction}

\begin{theo} \label{Main Theorem 1} For any strict double category $\mathbb B$, there is an equivalence of categories
\[ \dfib(\mathbb B) \simeq \lax(\mathbb B^{op},\spn)
\]
induced by the double category of elements construction.
\end{theo}
\begin{proof} It remains to check the naturality of the isomorphisms in Constructions \ref{equiv iso 1} and \ref{equiv iso 2}. If $H$ and $G$ are lax functors with a transformation $\tau\colon H\to G$, the diagram 
$$\xymatrix{ 
H \ar[d]_\tau \ar[r]^{\eta} & F_{\dblelt(H)}\ar[d] \\
G \ar[r]_{\eta} & F_{\dblelt(G)}
}$$
commutes because indexing commutes with applying the components of $\tau$. Naturality in $P$ also follows. Let $H\colon P\to Q$ be a morphism of double fibrations $P\colon\mathbb E\to\mathbb B$ and $Q\colon\mathbb G\to\mathbb B$. For naturality, it is required that the square
$$\xymatrix{
\dblelt(F_P) \ar[d]_{\dblelt(F_H)} \ar[r]^{\;\;\epsilon} & \mathbb E \ar[d]^{H} \\
\dblelt(F_Q) \ar[r]_{\;\;\epsilon} & \mathbb G
}$$
commutes. Chasing an object $(D,X)$ or a proarrow $(m,s)$ around each side of the square, the result either way is $HX$ in the former case and is $Hs$ in the latter case. Thus, to check are the arrows and cells. Given an arrow $f\colon (C,X)\to (D,Y)$, the counter-clockwise direction gives the left arrow below whereas the clockwise direction gives the one on the right:
\[ !\colon f^*Hy \to y\qquad\qquad  H(!)\colon Hf^*y\to Hy
\]
These are strictly equal, however, by uniqueness of lifts. A formally similar argument works to show that the square commutes also at the level of cells. \end{proof}

\section{Monoids, Modules and Multimodulations}

The main representation theorem of the paper asserts an equivalence of certain virtual double categories.  To define these first recall the definition, presented here in the form of \cite{FrameworkGenMulticatz}.  Virtual double categories have been known under the name ``$\mathbf{f.c.}$-multicategories," that is, as an example of a ``generalized multicategory," in this case relative to the free-category monad in \cite{Operads}.

\subsection{Virtual Double Categories}

\begin{define}[See \S 2.1 of \cite{FrameworkGenMulticatz}] \label{virt dbl cat defn}  A \textbf{virtual double category} $\mathbb D$ consists of an underlying category $\mathbb D_0$, giving the objects and morphisms of $\mathbb D$, together with, for any two objects $C$ and $D$, a class of proarrows $v\colon C\slashedrightarrow D$ and for any ``arity" $k$ multicells of the form
$$\xymatrix{
A_0 \ar@{}[drrr]|{\mu} \ar[d]_f \ar[r]^{m_1}|-@{|} & A_1 \ar[r]^{m_2}|-@{|} & \cdots \ar[r]^{m_k}|-@{|} & A_k \ar[d]^g \\
B_0 \ar[rrr]_{n}|-@{|} & & & B_1
}$$
all subject to the unit, composition and associativity axioms, as detailed in the reference. The list of proarrows $(m_1, m_2, \dots, m_k)$ is a ``$k$-ary multisource." The definition allows nullary multisources with $k=0$. A \textbf{functor} of virtual double categories $F\colon\mathbb C\to \mathbb D$ sends objects to objects, arrows to arrows, proarrows to proarrows and multicells to multicells in such a way as to preserve domains, codomains, multisources (so arities in particular), targets, identities and compositions.
\end{define}

\begin{example} Every double category is a virtual double category by forgetting external composition. In particular, for $\C$ with pullbacks $\Span(\C)$ is a virtual double category.
\end{example}

\begin{define}[Cf. \S 5.1 \cite{FrameworkGenMulticatz}] A \textbf{unit proarrow} for an object $D\in\mathbb D$ in a virtual double category is a proarrow $u_D\colon D\slashedrightarrow D$ and a nullary opcartesian cell
$$\xymatrix{
\ar@{}[dr]|{\Downarrow} D \ar@{=}[r] \ar@{=}[d]& D \ar@{=}[d] \\
D \ar[r]_{u_D}|-@{|} & D.
}$$
A virtual double category \textbf{has units} if it is equipped with a choice of unit for each object.
\end{define}

\begin{remark} The unit proarrow of any honest double category is a unit in this sense. Existence of units is part of the requirements for a \emph{virtual equipment} as described in \cite{FrameworkGenMulticatz}. The full structure will not be needed here. All the interesting examples possess all units, so without further ado all virtual double categories will be assumed to have units. Functors of virtual double categories will be assumed to be \textbf{normal} in the sense that they preserve nullary opcartesian cells.
\end{remark}

\begin{remark} The universal property of any unit $u_D$ on a given object $D$ implies that $\mathbb D$ possesses generic multicells 
$$\xymatrix{ 
\ar@{}[drrr]|{\Downarrow}\cdot \ar[r]^{u_D}|-@{|} \ar@{=}[d] & \cdot \ar[r]^{u_D}|-@{|} & \cdots\ar[r]^{u_D}|-@{|} & \cdot \ar@{=}[d] \\
\cdot\ar[rrr]_{u_D}|-@{|} & & & \cdot
}$$
of any arity $k$. These are given from the unique factorization. Moreover, by uniqueness of these factorizations, a correct multicomposite of any such cells gives the generic multicell of the proper arity determined in this way.  \end{remark}

\begin{example}[Terminal Object] \label{terminal object in virt dbl cats} The terminal object $\mathbf 1$ in $\vdbl$ is peculiar and is needed later. It has a single object $\bullet$, an identity arrow $1_\bullet$ and an ``identity proarrow" $u_\bullet$. Whereas one might expect there to be only a single multicell $u_\bullet \Rightarrow u_\bullet$, in fact required are generic multicells with multisources of all arities $k$
$$\xymatrix{
\bullet \ar@{}[drrr]|{\mu_k} \ar[d]_{1_\bullet} \ar[r]^{u_\bullet}|-@{|} & \bullet \ar[r]^{u_\bullet}|-@{|} & \cdots \ar[r]^{u_\bullet}|-@{|} & \bullet \ar[d]^{1_\bullet}\\
\bullet \ar[rrr]_{u_\bullet}|-@{|} & & & \bullet
}$$
as otherwise the natural definitions on objects, arrows and proarrows will not extend to a unique functor of virtual double categories $\mathbb D\to\mathbf 1$. Multi-composition is defined to give the generic multicell with the appropriate arity. Notice, then, that in particular $\mathbf 1$ has units by fiat.
\end{example}

\begin{example} A \textbf{point} of a virtual double category is a (normalized) functor $D\colon\mathbf 1\to\mathbb D$. A point thus consists of an object $D$, its identity arrow $1_D$, its unit proarrow $u_D$ and the corresponding generic multicells.
\end{example}

A notion of transformation gives the 2-categorical structure on virtual double categories.

\begin{define} \label{def: transf of funct of virt dble cats} Let $F,G\colon\mathbb C\rightrightarrows \mathbb D$ denote functors of virtual double categories. A \textbf{transformation} $\tau\colon F\to G$ assigns to each object $C$ of $\mathbb C$ an arrow $\tau_C\colon FC\to GC$ and to each proarrow $m\colon C\slashedrightarrow D$ a cell 
$$\xymatrix{
FC \ar@{}[dr]|{\tau_m} \ar[d]_{\tau_C} \ar[r]^{Fm}|-@{|} & FD \ar[d]^{\theta_D} \\
GC \ar[r]_{Gm}|-@{|} & GD
}$$
in such a way that
\begin{enumerate}
\item{}[Arrow Naturality] for each arrow $f\colon C\to D$, the square
$$\xymatrix{
FC \ar[d]_{\tau_C} \ar[r]^{Ff} & FD \ar[d]^{\theta_D} \\
GC \ar[r]_{Gf} & GD
}$$
commutes; and
\item{}[Cell Naturality] for each multicell
$$\xymatrix{
\cdot \ar@{}[drrr]|{\mu} \ar[d]_f \ar[r]^{m_1}|-@{|} & \cdot \ar[r]^{m_2}|-@{|} & \cdots \ar[r]^{m_k}|-@{|} & \cdot\ar[d]^g\\
\cdot \ar[rrr]_n|-@{|} & & & \cdot 
}$$
the composed multicells on either side of 
$$\xymatrix{
\cdot \ar@{}[dr]|{\tau_{m_1}} \ar[d] \ar[r]^{Fm_1}|-@{|} & \cdot \ar@{}[dr]|{\tau_{m_2}} \ar[d] \ar[r]^{Fm_2}|-@{|} & \cdots\ar@{}[dr]|{\tau_{m_k}} \ar[r]^{Fm_k}|-@{|} & \cdot \ar[d] & & \cdot \ar@{}[drrr]|{F\mu} \ar[d] \ar[r]^{Fm_1}|-@{|} & \cdot \ar[r]^{Fm_2}|-@{|} & \cdots \ar[r]^{Fm_k}|-@{|} & \cdot \ar[d] \\
\cdot \ar@{}[drrr]|{G\mu} \ar[d] \ar[r]_{Gm_1}|-@{|} & \cdot \ar[r]_{Gm_2}|-@{|} & \cdots \ar[r]_{Gm_k}|-@{|} & \cdot \ar[d] & = & \cdot \ar@{}[drrr]|{\tau_n} \ar[d] \ar[rrr]_{Fn}|-@{|} & & & \cdot \ar[d] \\
\cdot \ar[rrr]_{Gn}|-@{|} & & & \cdot & & \cdot \ar[rrr]_{Gn}|-@{|} & & & \cdot
}$$
are equal.
\end{enumerate}
Denote the 2-category of virtual double categories, their functors and transformations by $\vdbl$.
\end{define}

\begin{prop} \label{pullbacks} $\vdbl$ has (strict 2-)pullbacks.
\end{prop}
\begin{proof} Given two functors $F\colon\mathbb A\to\mathbb C$ and $G\colon\mathbb B\to\mathbb C$, the pullback has as its underlying category 
\[ (\mathbb A\times_{\mathbb C}\mathbb B)_0 = \mathbb A_0\times_{\mathbb B_0}\mathbb C_0.
\]
Proarrows are pairs $(m,n)$ for a proarrow $m$ of $\mathbb A$ and one $n$ of $\mathbb B$ satisfying $Fm=Gn$. Similarly, multicells are pairs $(\mu,\nu)$ with $\mu$ in $\mathbb A$ and $\nu$ in $\mathbb B$ such that $F\mu=G\nu$. Composition uses composition in $\mathbb A$ and $\mathbb B$. This is a virtual double category fitting into a commutative square
$$\xymatrix{
 \mathbb A\times_{\mathbb C}\mathbb B \ar[d]_{d_0} \ar[r]^{\;\;\;\;d_1} & \mathbb B \ar[d]^G \\ 
\mathbb A \ar[r]_F & \mathbb C
}$$
with the expected 2-categorical universal property \cite{StreetLimitsIndexed}..  \end{proof}

 Further limits of a 2-categorical variety abound in $\vdbl$. The comma category is of interest. These are well-known in $\cat$ (e.g. \S I.6 \cite{MacLane}). The formal abstraction and and elementary phrasing of its universal property in an arbitrary 2-category appear in \S 1 of \cite{StreetFibrations}.

\begin{prop} The 2-category $\vdbl$ has comma objects.
\end{prop}
\begin{proof} Given functors $F\colon\mathbb A\to\mathbb C$ and $G\colon\mathbb B\to\mathbb C$, the (purported) comma $F/G$ has its underlying 1-category as $(F/G)_0 = F_0/G_0$. Proarrows are triples $(m, \alpha, n)$ with $m$ and $n$ proarrows of $\mathbb A$ and $\mathbb B$, respectively, and $\alpha$ a cell $\alpha\colon Fm \Rightarrow Gn$ of $\mathbb C$. A multicell is thus a pair of multicells
$$\xymatrix{
\cdot \ar@{}[drrr]|{\mu} \ar[d] \ar[r]^{m_1}|-@{|} & \cdot \ar[r]^{m_2}|-@{|} & \cdots \ar[r]^{m_k}|-@{|} & \cdot\ar[d] && \cdot \ar@{}[drrr]|{\nu} \ar[d] \ar[r]^{n_1}|-@{|} & \cdot \ar[r]^{n_2}|-@{|} & \cdots \ar[r]^{n_k}|-@{|} & \cdot\ar[d]\\
\cdot \ar[rrr]_p|-@{|} & & & \cdot && \cdot \ar[rrr]_q|-@{|} & & & \cdot 
}$$
from $\mathbb A$ and $\mathbb B$, respectively, of the same arity and satisfying the equation
$$\xymatrix{
\cdot \ar@{}[dr]|{\alpha_1} \ar[d] \ar[r]^{Fm_1}|-@{|} & \cdot \ar@{}[dr]|{\alpha_2} \ar[d] \ar[r]^{Fm_2}|-@{|} & \cdots\ar@{}[dr]|{\alpha_k} \ar[r]^{Fm_k}|-@{|} & \cdot \ar[d] & & \cdot \ar@{}[drrr]|{F\mu} \ar[d] \ar[r]^{Fm_1}|-@{|} & \cdot \ar[r]^{Fm_2}|-@{|} & \cdots \ar[r]^{Fm_k}|-@{|} & \cdot \ar[d] \\
\cdot \ar@{}[drrr]|{G\nu} \ar[d] \ar[r]_{Gn_1}|-@{|} & \cdot \ar[r]_{Gn_2}|-@{|} & \cdots \ar[r]_{Gn_k}|-@{|} & \cdot \ar[d] & = & \cdot \ar@{}[drrr]|{\beta} \ar[d] \ar[rrr]_{Fp}|-@{|} & & & \cdot \ar[d] \\
\cdot \ar[rrr]_{Gq}|-@{|} & & & \cdot & & \cdot \ar[rrr]_{Gq}|-@{|} & & & \cdot
}$$
Composition is given by that in $\mathbb A$ and $\mathbb B$. So defined, $F/G$ comes with evident projection functors to $\mathbb A$ and $\mathbb B$. The expected transformation 
$$\xymatrix{
\ar@{}[dr]|{\Rightarrow} F/G\ar[d]_{d_0} \ar[r]^{d_1} & \mathbb B \ar[d]^G \\ 
\mathbb A \ar[r]_F & \mathbb C
}$$
has components $f$ for objects $(A, f, B)$ and $\alpha$ for proarrows $(m,\alpha, n)$. It satisfies the required 2-categorical universal property in the reference by construction.
\end{proof}

\begin{remark} This makes $\vdbl$ into a ``representable 2-category" \cite{GrayFormalCats} and \cite{StreetFibrations}.
\end{remark}

\begin{example}[Slice Virtual Double Category] \label{slice} Let $D$ denote an object in a virtual double category $\mathbb D$. The \textbf{slice} virtual double category over $D$ is defined as the comma $1/D$
$$\xymatrix{
\ar@{}[dr]|{\Rightarrow} 1/D\ar[d] \ar[r]^{d_1} & \mathbf 1 \ar[d]^D \\ 
\mathbb D \ar[r]_1 & \mathbb D
}$$
Denote this as usual by $\mathbb D/D$. The coslice is defined analogously.
\end{example}

\subsection{Virtual Double Category Structure on Double Presheaves}

The virtual double category structure on $\mathbf{Lax}(\mathbb B^{op},\spn)$ will be given by taking so-called ``modules" as proarrows and ``multimodulations" as the multicells. Here we revisit the the definitions for lax functors between arbitrary double categories. A path of proarrows is a sequence $\mathbf m = (m_1, \dots, m_k)$ such that, reading left to right, the target of one proarrow is the source of the next. The external composite is denoted by $[\mathbf m]$.

\begin{define}[Cf. \S 3.2 \cite{YonedaThyDblCatz}] \label{def: module}  A \textbf{module} between lax functors $M\colon F\slashedrightarrow G\colon \mathbb A\rightrightarrows \mathbb B$ of double categories assigns
\begin{enumerate}
\item to each proarrow $m\colon A \slashedrightarrow B$ of $\mathbb A$, a proarrow $Mm\colon FA\slashedrightarrow GB$ of $\mathbb B$;
\item to each cell of $\mathbb A$ as at left, one of $\mathbb B$ as at right
$$\xymatrix{ \ar @{} [dr] |{\theta} A \ar[r]^{m}|-@{|} \ar[d]_{f} & B \ar[d]^{g} & & & & \ar @{} [dr] |{M\theta} FA \ar[r]^{Mm}|-@{|} \ar[d]_{Ff} & GB \ar[d]^{Gg} \\  C \ar[r]_{n}|-@{|} & D & & & & FC \ar[r]_{Mn}|-@{|} & GD \\
}$$
\item for each pair of proarrows $m\colon A\slashedrightarrow B$ and $n\colon B\slashedrightarrow C$, action multicells of $\mathbb B$
$$\xymatrix{ \ar @{} [drr] |{\lambda_{m,n}} \cdot \ar[r]^{Fm}|-@{|} \ar@{=}[d]_{}& \cdot \ar[r]^{Mn}|-@{|} & \cdot  \ar@{=}[d] & & &  \ar @{} [drr] |{\rho_{m,n}} \cdot \ar[r]^{Mm}|-@{|} \ar@{=}[d]_{}& \cdot \ar[r]^{Gn}|-@{|} & \cdot  \ar@{=}[d]    \\
\cdot \ar[rr]_{M(n\otimes m)}|-@{|}& & \cdot & & & \cdot \ar[rr]_{M(n\otimes m)}|-@{|}& & \cdot \\ 
}$$
\end{enumerate}
in such a way that the following axioms hold.
\begin{enumerate}
\item{}[Functoriality] For any (internal) composite of monocells $\beta\alpha$ and any proarrow $m$, the equations $M\beta M\alpha = M(\beta\alpha)$ and $M1_m= 1_{Mm}$ hold;
\item{}[Naturality] For any external composite $\beta\otimes \alpha$, the equalities
$$\xymatrix{ \ar@{}[drr]|{\lambda} \cdot \ar[r]^{Fm}|-@{|} \ar@{=}[d]& \cdot \ar[r]^{Mn}|-@{|}& \cdot \ar@{=}[d]& & &  & \ar@{}[dr]|{F\alpha} \cdot \ar[r]^{Fm}|-@{|} \ar[d]  & \ar@{}[dr]|{M\beta} \cdot \ar[r]^{Mn}|-@{|} \ar[d]& \cdot\ar[d] \\ 
\ar@{}[drr]|{M(\beta\otimes \alpha)}\cdot \ar[rr]_{}|-@{|} \ar[d] &  & \cdot \ar[d] & & =& & \ar@{}[drr]|{\lambda}\cdot  \ar[r]|-@{|} \ar@{=}[d]& \cdot \ar[r]|-@{|}& \cdot \ar@{=}[d] \\ 
\cdot \ar[rr]_{M(q\otimes p)}|-@{|} &  & \cdot & & & & \cdot \ar[rr]_{M(q\otimes p)}|-@{|}  &  & \cdot
}$$
and
$$\xymatrix{ \ar@{}[drr]|{\rho} \cdot \ar[r]^{Mm}|-@{|} \ar@{=}[d]& \cdot \ar[r]^{Gn}|-@{|}& \cdot \ar@{=}[d]& & &  & \ar@{}[dr]|{M\alpha} \cdot \ar[r]^{Mm}|-@{|} \ar[d]  & \ar@{}[dr]|{G\beta} \cdot \ar[r]^{Gn}|-@{|} \ar[d]& \cdot\ar[d] \\ 
\ar@{}[drr]|{M(\beta\otimes \alpha)}\cdot \ar[rr]_{}|-@{|} \ar[d] &  & \cdot \ar[d] & & =& & \ar@{}[drr]|{\rho}\cdot  \ar[r] \ar@{=}[d]& \cdot \ar[r]& \cdot \ar@{=}[d] \\ 
\cdot \ar[rr]_{M(q\otimes p)}|-@{|} &  & \cdot & & & & \cdot \ar[rr]_{M(q\otimes p)}|-@{|}  &  & \cdot
}$$
both hold.
\item{}[Associativity]  For any composable sequence of proarrows $(m,n,p)$ of $\mathbb A$, the equalities (modulo suppressed associativity isomorphisms)
$$\xymatrix{ \ar@{}[dr]|{1}\cdot \ar@{=}[d] \ar[r]^{Fm}|-@{|} & \ar@{}[drr]|{\lambda} \cdot \ar@{=}[d] \ar[r]^{Fn}|-@{|} & \cdot \ar[r]^{Mw}|-@{|} & \cdot \ar@{=}[d] & & & & \ar@{}[drr]|{F\gamma}\cdot \ar@{=}[d] \ar[r]^{Fm}|-@{|} & \cdot \ar[r]^{Fn}|-@{|} & \cdot \ar@{}[dr]|{1} \ar@{=}[d]  \ar[r]^{Mw}|-@{|} & \cdot \ar@{=}[d] \\
\ar@{}[drrr]|{\lambda} \cdot \ar@{=}[d] \ar[r]|-@{|} & \cdot \ar[rr]|-@{|} &   & \cdot \ar@{=}[d] & & = & & \ar@{}[drrr]|{\lambda} \cdot \ar@{=}[d] \ar[rr]|-@{|} &   & \cdot \ar[r]|-@{|} & \cdot \ar@{=}[d]\\
\cdot \ar[rrr]_{M(p\otimes (n\otimes m))}|-@{|} &  &  & \cdot  & & & & \cdot \ar[rrr]_{M((p\otimes n)\otimes m)}|-@{|}&  &  & \cdot\\
}$$
and
$$\xymatrix{ \ar@{}[drr]|{\rho} \cdot \ar[r]^{Mm}|-@{|} \ar@{=}[d] & \cdot \ar[r]^{Gn}|-@{|} & \ar@{}[dr]|{1} \cdot \ar[r]^{Gp}|-@{|} \ar@{=}[d] & \cdot \ar@{=}[d] & & & & \ar@{}[dr]|{1} \cdot \ar[r]^{Mm}|-@{|} \ar@{=}[d]  & \ar@{}[drr]|{G\gamma} \cdot \ar[r]^{Gn}|-@{|} \ar@{=}[d] &  \cdot \ar[r]^{Gp}|-@{|} & \cdot \ar@{=}[d] \\
\ar@{}[drrr]|{\rho} \cdot \ar@{=}[d] \ar[rr]|-@{|} &  & \cdot \ar[r]|-@{|} & \cdot \ar@{=}[d]  & & = & & \ar@{}[drrr]|{\rho}\cdot \ar@{=}[d] \ar[r]|-@{|} & \cdot \ar[rr]|-@{|} &  & \cdot \ar@{=}[d]\\
\cdot \ar[rrr]_{M(p\otimes (n\otimes m))}|-@{|} &  &  & \cdot  & & & & \cdot \ar[rrr]_{M((p\otimes n)\otimes m)}|-@{|} &  &  & \cdot\\
}$$
are valid.
\item{}[Compatibility] The composites
$$\xymatrix{ \ar@{}[drr]|{\lambda} \cdot \ar[r]^{Fm}|-@{|} \ar@{=}[d] & \cdot \ar[r]^{Mn}|-@{|} & \ar@{}[dr]|{1} \cdot \ar[r]^{Gp}|-@{|} \ar@{=}[d] & \cdot \ar@{=}[d] & & & & \ar@{}[dr]|{1} \cdot \ar[r]^{Fm}|-@{|} \ar@{=}[d]  & \ar@{}[drr]|{\rho} \cdot \ar[r]^{Mn}|-@{|} \ar@{=}[d] &  \cdot \ar[r]^{Gp}|-@{|} & \cdot \ar@{=}[d] \\
\ar@{}[drrr]|{\rho} \cdot \ar@{=}[d] \ar[rr]|-@{|} &  & \cdot \ar[r]|-@{|} & \cdot \ar@{=}[d]  & & = & & \ar@{}[drrr]|{\lambda}\cdot \ar@{=}[d] \ar[r]|-@{|} & \cdot \ar[rr]|-@{|} &  & \cdot \ar@{=}[d]\\
\cdot \ar[rrr]_{M(p\otimes (n\otimes m))}|-@{|} &  &  & \cdot  & & & & \cdot \ar[rrr]_{M((p\otimes n)\otimes m)}|-@{|} &  &  & \cdot\\
}$$
are equal.
\item{}[Unit] For any proarrow $m\colon A\slashedrightarrow B$ of $\mathbb A$, the composites
$$\xymatrix{ \ar@{}[dr]|{\gamma} \cdot \ar[r]^{u_{FA}}|-@{|} \ar@{=}[d]& \ar@{}[dr]|{1} \cdot \ar[r]^{Mm}|-@{|} \ar@{=}[d] & \cdot \ar@{=}[d] & & &  & \ar@{}[dr]|{1} \cdot \ar[r]^{Mm}|-@{|} \ar[d]  & \ar@{}[dr]|{\gamma} \cdot \ar[r]^{u_{GB}}|-@{|} \ar[d]& \cdot\ar[d] \\ 
\ar@{}[drr]|{\lambda} \cdot \ar[r]_{Fu_A}|-@{|} \ar[d] & \cdot \ar[r]|-@{|} & \cdot \ar[d] & & \text{and}& & \ar@{}[drr]|{\rho}\cdot  \ar[r]|-@{|} \ar@{=}[d]& \cdot \ar[r]_{Gu_B}|-@{|}& \cdot \ar@{=}[d] \\ 
\cdot \ar[rr]_{M(m\otimes u_A)}|-@{|} &  & \cdot & & & & \cdot \ar[rr]_{M(u_B\otimes m)}|-@{|}  &  & \cdot
}$$
are equal to the respective canonical composition multicells for $(m_{FA},Mm)$ and $(Mu,m_{GA})$.
\end{enumerate}
\end{define}

Multicells are given by the notion of a ``multimodulation," recalled next. A path of cells $\theta\colon (\theta_1,\dots, \theta_k)$ a sequence of cells $\theta_1\colon m_1\Rightarrow n_1,\dots, \theta_k\colon m_k\Rightarrow n_k$ has the target of a given cell equal to the source of the next. Externally composable sequences of modules $M_1,\dots, M_k$ between lax functors will be thought of as proarrows $F^i\slashedrightarrow F^{i+1}$ from a lax functor $F^i$ to one $F^{i+1}$, notated with superscripts so as not to confuse these with the components of the functors.

\begin{define}[See \S 4.1 of \cite{YonedaThyDblCatz} and \S 1.2.3 of \cite{ModuleComposition}] \label{multimodulation defn} A \textbf{multimodulation}
$$\xymatrix{ 
\ar@{}[drrr]|{\mu}\cdot\ar[d]_{\tau} \ar[r]^{M_1}|-@{|} & \cdot\ar[r]^{M_2}|-@{|} &\cdots \ar[r]^{M_k}|-@{|} & \cdot \ar[d]^{\sigma} \\
\cdot \ar[rrr]_{N}|-@{|}&&& \cdot
}$$
from modules $M_i$ to $N$ with source $\tau$ and target $\sigma$ assigns to each path $\mathbf m = (m_1, \dots, m_k)$ of proarrows of $\mathbb A$, a multicell
$$\xymatrix{ 
\ar@{}[drrr]|{\mu_{\mathbf m}}\cdot\ar[d]_{\tau} \ar[r]^{M_1m_1}|-@{|} & \cdot\ar[r]^{M_2m_2}|-@{|} &\cdots \ar[r]^{M_km_k}|-@{|} & \cdot \ar[d]^{\sigma} \\
\cdot \ar[rrr]_{N{[\mathbf m]}}|-@{|}&&& \cdot
}$$
in such a way that the following axioms are satisfied.
\begin{enumerate}
\item{}[Naturality] for any path of cells $\theta_1\colon m_1\Rightarrow n_1,\dots, \theta_k\colon m_k\Rightarrow n_k$, the two composites 
$$\xymatrix{
\cdot \ar[d] \ar@{}[dr]|{M_1\theta_1} \ar[r]^{M_1m_1}|-@{|} & \cdot \ar[d]\ar@{}[dr]|{M_2\theta_2} \ar[r]^{M_2m_2}|-@{|} & \cdots \ar@{}[dr]|{M_k\theta_k} \ar[r]^{M_km_k} &\ar[d] \cdot  && \cdot \ar@{}[drrr]|{\mu_{\mathbf m}} \ar[d] \ar[r]^{M_1m_1}|-@{|} & \cdot \ar[r]^{M_2m_2}|-@{|} & \cdots \ar[r]^{M_km_k} &\ar[d] \cdot\\
\cdot \ar[d] \ar@{}[drrr]|{\mu_{\mathbf n}} \ar[r]_{M_1n_1}|-@{|} & \cdot \ar[r]_{M_2n_2} & \cdots \ar[r]_{M_kn_k} &\cdot \ar[d] & = & \cdot\ar[d] \ar@{}[drrr]|{N[\theta]} \ar[rrr]_{N[\mathbf m]}|-@{|} &&& \cdot\ar[d]\\
\cdot \ar[rrr]_{N[\mathbf n]} &&&\cdot & & \cdot \ar[rrr]_{N[\mathbf n]}|-@{|} &&&\cdot 
}$$
are equal.
\item The following equivariance axioms are satisfied:
\begin{enumerate}
\item{}[Left and Right Equivariance] For any paths of proarrows $y\mathbf m = (m_1,\dots, m_k,y)$ and $\mathbf mx = (x, m_1, \dots, m_k)$, the composites on either side of 
$$\xymatrix{ 
\ar@{}[dr]|{\tau_x}\cdot \ar[r]^{F^0x}|-@{|} \ar[d]_{} & \ar@{}[drrr]|{\mu_{[\mathbf m]}}\cdot \ar[r]^{M_1m_1}|-@{|} \ar[d]^{} & \cdot \ar[r]^{M_2m_2}|-@{|} &\cdots \ar[r]^{M_km_k}|-@{|}& \cdot \ar[d]^{} & & \ar@{}[drr]|{\lambda} \cdot \ar@{=}[d] \ar[r]^{F^0x}|-@{|} & \cdot \ar[r]^{M_1m_1}|-@{|} & \ar@{}[dr]|{1} \cdot \ar@{=}[d] \ar[r]^{M_2m_2}|-@{|} & \ar@{}[dr]|{1}\cdots \ar[r]^{M_km_k}|-@{|} & \cdot\ar@{=}[d] \\
\ar@{}[drrrr]|{\lambda}\cdot \ar[r]_{G^0x}|-@{|} \ar@{=}[d] &\cdot \ar[rrr]_{N[\mathbf m]}|-@{|}  &  &  & \cdot \ar@{=}[d] & = &\ar@{}[drrrr]|{\mu_{[\mathbf mx]}} \cdot \ar[rr]_{M_1(m_1\otimes x)}|-@{|} \ar[d]_{} &  & \cdot \ar[r]_{M_2m_2}|-@{|} & \cdots \ar[r]_{M_km_k}|-@{|}& \cdot \ar[d]^{} \\
\cdot \ar[rrrr]_{N[\mathbf mx]}|-@{|} & && & \cdot &&\cdot \ar[rrrr]_{N_{[\mathbf mx]}}|-@{|} & &&&\cdot 
}$$
and
$$\xymatrix{ 
\ar@{}[drrr]|{\mu_{[\mathbf m]}}\cdot \ar[d]_{}  \ar[r]^{M_1m_1}|-@{|}&\cdots \ar[r]^{}|-@{|} & \cdot \ar[r]^{M_km_k}|-@{|}  & \ar@{}[dr]|{\tau_y} \cdot \ar[d] \ar[r]^{F^ny}|-@{|} & \cdot \ar[d]^{} && \ar@{}[dr]|{1}\cdot \ar[r]^{M_1m_1}|-@{|} \ar@{=}[d]& \ar@{}[dr]|{1} \cdots \ar[r]^{}|-@{|}  & \ar@{}[drr]|{\rho}\cdot\ar@{=}[d] \ar[r]^{M_km_k}|-@{|} &\cdot\ar[r]^{F^ny}|-@{|}&\cdot \ar@{=}[d]  \\
\ar@{}[drrrr]|{\rho}\cdot \ar@{=}[d] \ar[rrr]_{N[\mathbf m]}|-@{|} &&   & \cdot \ar[r]_{G^ny}|-@{|} &\cdot \ar@{=}[d] & = & \ar@{}[drrrr]|{\mu[y\mathbf m]}\cdot \ar[d]_{} \ar[r]_{M_1m_1}|-@{|} &\cdots \ar[r]_{}|-@{|} &\cdot \ar[rr]_{M_k(y\otimes m_k)}|-@{|} && \cdot\ar[d]^{}\\
\cdot \ar[rrrr]_{N[y\mathbf m]}|-@{|} &&& & &&\cdot \ar[rrrr]_{N[y\mathbf m]}|-@{|} && & &\cdot
}$$
are equal.
\item{}[Inner Equivariance] For any path of proarrows $(m_1,\dots, m_i,x,m_{i+1},\dots, m_k)$, the composite
$$\xymatrix{ \ar@{}[dr]|{1}\cdot \ar@{=}[d] \ar[r]^{M_1m_1}|-@{|} &\ar@{}[dr]|{1}\cdots\ar[r]^{}|-@{|} &\ar@{}[drr]|{\rho}\cdot\ar@{=}[d]\ar[r]^{M_im_i}|-@{|}&\cdot\ar[r]^{F^ix}|-@{|}&\ar@{}[dr]|{1}\cdot\ar@{=}[d]\ar[r]^{}|-@{|}&\ar@{}[dr]|{1}\cdots\ar[r]^{M_km_k}|-@{|}&\cdot \ar@{=}[d]\\
\ar@{}[drrrrrr]|{\mu_{[\mathbf m]}}\cdot\ar[d]_{}\ar[r]^{}|-@{|}&\cdots\ar[r]^{}|-@{|}&\cdot\ar[rr]^{}|-@{|}&&\cdot\ar[r]^{}|-@{|}&\cdots\ar[r]^{}|-@{|}&\cdot \ar[d]^{} \\
\cdot \ar[rrrrrr]_{N[\mathbf m]}|-@{|} &&&&&& \cdot
}$$
is equal to
$$\xymatrix{ \ar@{}[dr]|{1}\cdot \ar@{=}[d] \ar[r]^{M_1m_1}|-@{|} &\ar@{}[dr]|{1}\cdots\ar[r]^{}|-@{|} &\ar@{}[drr]|{\lambda}\cdot\ar@{=}[d]\ar[r]^{F^ix}|-@{|}&\cdot\ar[r]^{M_{i+1}m_{i+1}}|-@{|}&\ar@{}[dr]|{1}\cdot\ar@{=}[d]\ar[r]^{}|-@{|}&\ar@{}[dr]|{1}\cdots\ar[r]^{M_km_k}|-@{|}&\cdot \ar@{=}[d]\\
\ar@{}[drrrrrr]|{\mu_{[\mathbf m]}}\cdot\ar[d]_{}\ar[r]^{}|-@{|}&\cdots\ar[r]^{}|-@{|}&\cdot\ar[rr]^{}|-@{|}&&\cdot\ar[r]^{}|-@{|}&\cdots\ar[r]^{}|-@{|}&\cdot \ar[d]^{} \\
\cdot \ar[rrrrrr]_{N[\mathbf m]}|-@{|} &&&&&& \cdot \\
}$$
for $i=1,\dots, k-1$.
\end{enumerate}
\end{enumerate}
\end{define}

\begin{theo}[Par\'e]  Double functors $F\colon \mathbb A\to\mathbb B$ and horizontal transformations, together with modules and their multi-modulations giving the proarrows and multicells, comprise a virtual double category denoted by $\dbllax(\mathbb A,\mathbb B)$.
\end{theo}
\begin{proof} See the lead-up to Theorem 1.2.5 of \cite{ModuleComposition}.  \end{proof}

\begin{remark} Our main interest is of course in the virtual double category $\dbllax(\mathbb B^{op},\spn)$. By the Theorem, in general, this is not a genuine double category. This is because composition of proarrows need not exist.  The paper \cite{ModuleComposition} is a dedicated study of this issue.  
\end{remark}

\subsection{Monoids and Modules}

The additional structure on $\dfib(\mathbb B)$ making it a virtual double category goes by well-known terminology from another context.  We already know that $\dfib(\mathbb B)$ could have been defined as $\cat(\dfib)/\mathbb B$. Another way of look at $\cat(\dfib)$ is that it is the category $\mon(\Span(\dfib))$ of monoids in spans in discrete fibrations. In general $\mon(\Span(\C))$ for a category $\C$ with finite limits is the underlying category of the virtual double category $\prof(\C) = \dblmod(\Span(\C))$ of modules in spans in $\C$ as in \cite{Operads} or \cite{FrameworkGenMulticatz}. So, we take the modules and their multicells from this context as the virtual double category structure on $\dfib(\mathbb B)$. Here are the definitions.

\begin{define}[Cf. \S 5.3.1 \cite{Operads} or \S 2.8 of \cite{FrameworkGenMulticatz}] \label{def: monoids and modules} Let $\mathbb D$ denote a virtual double category. Define its \textbf{virtual double category of monoids and modules}, denoted by $\dblmod(\mathbb D)$, by taking
\begin{enumerate}
\item objects: \textbf{monoids}, namely, triples $(r,\mu,\eta)$ consisting of a proarrow $r\colon A\slashedrightarrow A$ and cells 
$$\xymatrix{  \ar@{}[drr]|{\mu} \ar@{=}[d] \ar[r]^{r}|-@{|} A  & A \ar[r]^{r}|-@{|}  &  A \ar@{=}[d] & \ar@{}[dr]|{\eta} A \ar@{=}[r] \ar@{=}[d] & A \ar@{=}[d] \\
A \ar[rr]_{r}|-@{|} & & A & A \ar[r]_{r}|-@{|} & A \\
}$$
satisfying the usual axioms for a monoid, namely, the multiplication law $\mu(1,\mu) = \mu(\mu,1)$ and the unit laws $\mu(1,\eta)=1$ and $\mu(\eta,1)=1$;
\item arrows: monoid homomorphisms $(r,\mu,\eta) \to (s,\nu,\epsilon)$, namely, those pairs $(f,\phi)$ consisting of an arrow $f\colon A\to B$ and a cell
$$\xymatrix{ \ar@{}[dr]|{\phi} A \ar[d]_{f} \ar[r]^{r}|-@{|} & A\ar[d]^{f}\\
B \ar[r]_{s}|-@{|}& B\\
}$$
satisfying the unit axiom $\phi\eta = \epsilon f$ and multiplication axiom $\nu(\phi,\phi) = \phi\mu$.
\item proarrows: so-called \textbf{modules} $(r,\mu,\eta) \slashedrightarrow (s,\nu, \epsilon)$, namely, triples $(m,\lambda,\rho)$ with $m\colon A\slashedrightarrow B$ a proarrow and $\lambda$, $\rho$ left and right action cells
$$\xymatrix{ \ar@{}[drr]|{\lambda} A \ar[r]^{r}|-@{|} \ar@{=}[d] & A \ar[r]^{m}|-@{|} & B \ar@{=}[d] & \ar@{}[drr]|{\rho} A \ar@{=}[d] \ar[r]^{m}|-@{|} & B \ar[r]^{s}|-@{|} & B \ar@{=}[d] \\
A \ar[rr]_{m}|-@{|} &  & B & A \ar[rr]_{m}|-@{|} & & B  \\
}$$
satisfying the module axioms $\lambda(\mu, 1) = \lambda(1,\lambda)$ and $\rho(1,\mu) = \rho(\rho,1)$ for the multiplication and $\lambda(\eta,1) = 1$ and $\rho(1,\eta) = 1$ for the units; a sequence of modules consists of finitely many modules $(m_i,\lambda_i,\rho_i)$ for which $\src\,m_{i+1} = \tgt\,m_{i}$ and $s_{i+1} = r_{i}$ both hold;
\item multicells from a sequence of modules $(m_i, \lambda_i,\rho_i)$ to one $(n, \lambda, \rho)$ consist of those multicells in $\mathbb A$ 
$$\xymatrix{ \ar@{}[drr]|{\gamma} \cdot \ar[d]_{f} \ar[r]^{m_1}|-@{|} & \cdots \ar[r]^{m_p}|-@{|} & \cdot \ar[d]^{g} \\
\cdot \ar[rr]_{n}|-@{|} && \cdot \\
}$$
satisfying the equivariance axioms expressed by the equalities of composite cells:
\begin{enumerate}
\item{}[Left]
$$\xymatrix{ 
\ar@{}[dr]|{\phi}\cdot \ar[r]^{r_1}|-@{|} \ar[d]_{f} & \ar@{}[drrr]|{\gamma}\cdot \ar[r]^{m_1}|-@{|} \ar[d]^{f} & \cdot \ar[r]^{m_2}|-@{|} &\cdots \ar[r]^{m_p}|-@{|}& \cdot \ar[d]^{g} & & \ar@{}[drr]|{\lambda} \cdot \ar@{=}[d] \ar[r]^{r_1}|-@{|} & \cdot \ar[r]^{m_1}|-@{|} & \ar@{}[dr]|{1} \cdot \ar@{=}[d] \ar[r]^{m_2}|-@{|} & \ar@{}[dr]|{1}\cdots \ar[r]^{m_p}|-@{|} & \cdot\ar@{=}[d] \\
\ar@{}[drrrr]|{\lambda}\cdot \ar[r]_{s_1}|-@{|} \ar@{=}[d] &\cdot \ar[rrr]_{n}|-@{|}  &  &  & \cdot \ar@{=}[d] & = &\ar@{}[drrrr]|{\gamma} \cdot \ar[rr]_{m_1}|-@{|} \ar[d]_{f} &  & \cdot \ar[r]_{m_2}|-@{|} & \cdots \ar[r]_{m_p}|-@{|}& \cdot \ar[d]^{g} \\
\cdot \ar[rrrr]_{n}|-@{|} & && & \cdot &&\cdot \ar[rrrr]_{n}|-@{|} & &&&\cdot 
}$$
\item{}[Right]
$$\xymatrix{ 
\ar@{}[drrr]|{\gamma}\cdot \ar[d]_{f}  \ar[r]^{m_1}|-@{|}&\cdots \ar[r]^{m_{p-1}}|-@{|} & \cdot \ar[r]^{m_p}|-@{|}  & \ar@{}[dr]|{\psi} \cdot \ar[d]_{g} \ar[r]^{r_p}|-@{|} & \cdot \ar[d]^{g} && \ar@{}[dr]|{1}\cdot \ar[r]^{m_1}|-@{|} \ar@{=}[d]& \ar@{}[dr]|{1} \cdots \ar[r]^{m_{p-1}}|-@{|}  & \ar@{}[drr]|{\rho}\cdot\ar@{=}[d] \ar[r]^{m_p}|-@{|} &\cdot\ar[r]^{r_p}|-@{|}&\cdot \ar@{=}[d]  \\
\ar@{}[drrrr]|{\rho}\cdot \ar@{=}[d] \ar[rrr]_{n}|-@{|} &&   & \cdot \ar[r]_{s_q}|-@{|} &\cdot \ar@{=}[d] &=&\ar@{}[drrrr]|{\gamma}\cdot \ar[d]_{f} \ar[r]_{m_1}|-@{|} &\cdots \ar[r]_{m_{p-1}}|-@{|} &\cdot \ar[rr]_{m_p}|-@{|} && \cdot\ar[d]^{g}\\
\cdot \ar[rrrr]_{n}|-@{|} &&& & &&\cdot \ar[rrrr]_{n}|-@{|} && & &\cdot
}$$
\item{}[Inner]
$$\xymatrix{ \ar@{}[dr]|{1}\cdot \ar@{=}[d] \ar[r]^{m_1}|-@{|} &\ar@{}[dr]|{1}\cdots\ar[r]^{m_{i-1}}|-@{|} &\ar@{}[drr]|{\rho}\cdot\ar@{=}[d]\ar[r]^{m_i}|-@{|}&\cdot\ar[r]^{r_i}|-@{|}&\ar@{}[dr]|{1}\cdot\ar@{=}[d]\ar[r]^{m_{i+1}}|-@{|}&\ar@{}[dr]|{1}\cdots\ar[r]^{m_p}|-@{|}&\cdot \ar@{=}[d]\\
\ar@{}[drrrrrr]|{\gamma}\cdot\ar[d]_{f}\ar[r]^{}|-@{|}&\cdots\ar[r]^{}|-@{|}&\cdot\ar[rr]^{}|-@{|}&&\cdot\ar[r]^{}|-@{|}&\cdots\ar[r]^{}|-@{|}&\cdot \ar[d]^{g} \\
\cdot \ar[rrrrrr]_{n}|-@{|} &&&&&& \cdot
}$$
is equal to
$$\xymatrix{ \ar@{}[dr]|{1}\cdot \ar@{=}[d] \ar[r]^{m_1}|-@{|} &\ar@{}[dr]|{1}\cdots\ar[r]^{m_{i}}|-@{|} &\ar@{}[drr]|{\lambda}\cdot\ar@{=}[d]\ar[r]^{r_i}|-@{|}&\cdot\ar[r]^{m_{i+1}}|-@{|}&\ar@{}[dr]|{1}\cdot\ar@{=}[d]\ar[r]^{m_{i+2}}|-@{|}&\ar@{}[dr]|{1}\cdots\ar[r]^{m_p}|-@{|}&\cdot \ar@{=}[d]\\
\ar@{}[drrrrrr]|{\gamma}\cdot\ar[d]_{f}\ar[r]^{}|-@{|}&\cdots\ar[r]^{}|-@{|}&\cdot\ar[rr]^{}|-@{|}&&\cdot\ar[r]^{}|-@{|}&\cdots\ar[r]^{}|-@{|}&\cdot \ar[d]^{g} \\
\cdot \ar[rrrrrr]_{n}|-@{|} &&&&&& \cdot \\
}$$
for $i=1,\dots, p-1$.
\end{enumerate}
\end{enumerate}
Compositions and identities are given by those in $\mathbb A$.
\end{define}

\begin{prop} For any virtual double category $\mathbb D$, $\dblmod(\mathbb D)$ has units.
\end{prop}
\begin{proof} Any object in $\dblmod(\mathbb D)$ is a monoid. Its equipped multiplication gives the unit proarrow. For more see \S 5.5 of \cite{FrameworkGenMulticatz}. \end{proof}

\begin{remark} The definition above omits most of the diagrams and states just the equations out of space considerations. However, upon writing down all the diagrams, one might notice a formal similarity between these axioms and those for modules and multimodulations in $\dbllax(\mathbb B^{op},\spn)$. It is the point of the next section to show that this is in fact an equivalence of virtual double categories. First, however, let us consider some examples.
\end{remark}

\subsection{Examples}

Let $\mathscr C$ denote a category with finite limits. Then $\Span(\mathscr C)$ is a double category. As in \cite{FrameworkGenMulticatz}, denote $\dblmod(\Span(\mathscr C))$ by $\prof(\mathscr C)$. Several choices of $\C$ are of interest. Modules in $\prof(\C)$ are already known by well-established terminology. 

\begin{define}[Cf. \S 2.41 of \cite{TT}] \label{def: internal profunctor and transformation} Let $\mathscr C$ denote a category with finite limits; and let $\mathbb C$ and $\mathbb D$ denote internal categories. An \textbf{internal profunctor} $M\colon \mathbb C\slashedrightarrow\mathbb D$ is a module, i.e. a proarrow in $\prof(\mathscr C)$. A \textbf{multicell of internal profunctors} is thus a multicell as above.
\end{define}

\begin{example} The virtual double category $\prof(\set)$ is $\prof$. That is, a monoid in $\Span(set)$ is a category. A unit proarrow for such $\C$ is thus the span $\C_0 \leftarrow \C_1 \to \C_0$ formed from the domain and codomain maps with actions given by composition.
\end{example}

\begin{example} \label{example: prof cat} Letting $\C = \cat$ as a 1-category, $\prof(\cat)$ consists of usual double categories and double functors as the objects and arrows. Internal profunctors $M\colon\mathbb A\slashedrightarrow \mathbb B$ between double categories consist of a span $\mathbb A_0 \xleftarrow{\partial_0} \mathscr M \xrightarrow{\partial_1} \mathbb B_0$ and left and right action functors
\[ L\colon \mathbb A_1\times_{\mathbb A_0} \mathscr M \longrightarrow \mathscr M \qquad\qquad R\colon \mathscr M\times_{\mathbb B_0} \mathbb B_1 \longrightarrow \mathscr M
\]
satisfying the axioms above. A multicell of internal profunctors $(M_1,\dots, M_k) \Rightarrow N$ thus consists of a functor
\[ m\colon \mathscr M^1\times_{\mathbb A^1_0} \cdots \times_{\mathbb A^k_0}\mathscr M^k \longrightarrow \mathscr N 
\]
from the vertex of the composite of $(M_1,\dots, M_k)$, making a morphism of spans, and satisfying the various equivariance requirements as in the definition. Notice that owing to the peculiarities of the cell structure of $\mathbf 1$ as in Example \ref{terminal object in virt dbl cats}, a point $\mathbb D \colon \mathbf 1\to\prof(\cat)$ is a double category $\mathbb D$, with the identity double functor $1\colon \mathbb D\to\mathbb D$, the unit proarrow $u\colon\mathbb D\slashedrightarrow \mathbb D$, namely, the span formed by the external source and target functors with actions given by external composition, and finally multicells of all arities given by iterated external composition. This can all be generalized to $\cat(\C)$ for arbitrary $\C$ with finite limits. \end{example}

\begin{example} Let $\C = \cat^{\mathbf 2}$, the ``arrow category" of $\cat$ and consider $\prof(\cat^{\mathbf 2})$. A monoid is then a double functor, a morphism is a commutative square of double functors. A module consists of two modules in the former sense -- one between domains of the two double categories and one between the codomains; the vertices of these modules are related by a functor making a morphism of spans. Multicells have a similar ``two-tiered" structure. \end{example}

\begin{example}  Letting $\C=\dfib$, the virtual double category $\prof(\dfib)$ is the sub-virtual double category of the previous example where all the objects are not just double functors but are instead \emph{discrete double fibrations}. There is a codomain functor
\[ \cod\colon\prof(\dfib) \longrightarrow \prof(\cat)
\]
taking an object $P\colon \mathbb E\to\mathbb B$ its codomain $\mathbb B$ and every proarrow $M\colon P\slashedrightarrow Q$ to the module between double categories giving the codomains of $M$. Take $\prof(\dfib)/\mathbb B$ to be the pullback of $\cod$ in $\vdbl$ along the point $\mathbb B \colon \mathbf 1 \to \prof(\cat)$. 
\end{example}

\begin{define} \label{defn: virt dbl cat structure on ddfibs} The virtual double category of discrete double fibrations over a double category $\mathbb B$ is $\prof(\dfib)/\mathbb B$. Denote this by $\dblfib(\mathbb B)$.
\end{define}

\begin{remark} \label{description of internal profunctors and multicells} A module between discrete double fibrations $M\colon P\slashedrightarrow Q$ thus consists of a discrete fibration $M\colon\mathscr M\to\mathbb B_1$ and a morphism of spans
$$\xymatrix{
\mathbb E_0 \ar[d]_{P_0} & \mathscr M \ar[l]_{\partial_0} \ar[r]^{\partial_1} \ar[d]^{M} & \mathbb G_0 \ar[d]^{Q_0} \\
\mathbb B_0 & \mathbb B_1 \ar[l]^{\src} \ar[r]_{\tgt} & \mathbb B_0
}$$
and left and right actions functors making commutative squares
$$\xymatrix{ \mathbb E_1\times_{\mathbb E_0} \mathscr M \ar[r]^{\;\;\;\;L} \ar[d]_{P_1\times M} & \mathscr M \ar[d]^{M} & & \mathscr M\times_{\mathbb G_0}\mathbb G_1 \ar[r]^{\;\;\;\;R} \ar[d]_{M\times Q_1} & \mathscr M \ar[d]^M \\ 
\mathbb B_1\times_{\mathbb B_0}\mathbb B_1 \ar[r]_{\;\;\;\;\;-\otimes -} & \mathbb B_1 & & \mathbb B_1\times_{\mathbb B_0}\mathbb B_1 \ar[r]_{\;\;\;\;\;-\otimes -} & \mathbb B_1
}$$
that satisfy the action requirements as in the definition. A multicell $\mu\colon (M^1,\dots, M^k) \Rightarrow N$ between such modules consists of a functor $\mu$ making a commutative square
$$\xymatrix{ \mathscr M^1\times_{\mathbb E_0^1}\cdots \times_{\mathbb E_0^k}\mathscr M^k \ar[d] \ar[rr]^{\qquad\mu} & & \mathscr N \ar[d]^N \\
\mathbb B_1\times_{\mathbb B_0} \cdots \times_{\mathbb B_0}\mathbb B_1 \ar[rr]_{\qquad-\otimes -\cdots -\otimes -} & & \mathbb B_1
}$$
satisfying the equivariance requirements above. 
\end{remark}

\begin{remark} As in \S 3.9 in \cite{FrameworkGenMulticatz}, the mod-construction $\dblmod(-)$ defines an endo-2-functor $\dblmod(-)\colon\vdbl\to\vdbl$. Another way to look at the codomain functor the previous example is that it is induced from the codomain functor $\cod \colon\dfib \to\cat$, passing first through $\Span(-)$ and then $\dblmod(-)$.
\end{remark}

\section{The Full Representation Theorem}

This section extends the result of Theorem \ref{Main Theorem 1}, culminating in a proof that elements construction extends to an equivalence of virtual double categories
\[ \dblfib(\mathbb B) \simeq \dbllax(\mathbb B^{op},\spn) 
\]
This appears below as Theorem \ref{Main Theorem 2}.

\subsection{Extending the Elements Construction}

The elements functor of Lemma \ref{lemma: elements is a functor} extends to one between virtual double categories. Needed are assignments on modules and multimodulations.

\begin{construction}[Elements from a Module] \label{module to elements construction}  Let $M\colon F\slashedrightarrow G$ denote a module between lax double functors as in Definition \ref{def: module}.  Construct a category $\dblelt(M)$ in the following way.  Objects are pairs $(m,s)$ with $m\colon B\slashedrightarrow C$ a proarrow of $\mathbb B$ and $s\in Mm$.  A morphism $(m,s)\to (n,t)$ is a cell $\alpha$ with source $m$ and target $n$ for which the equation $M\alpha(t)=s$ holds.  So defined, $\dblelt(M)$ is a category since $M$ is strictly functorial on cells. Notice that there are thus projection functors 
\[\dblelt(F)_0  \xleftarrow{\partial_0} \dblelt(M) \xrightarrow{\partial_1} \dblelt(G)_0
\]
taking an object $(m,s)$ to $\partial_0(v,s)=(B,\partial_0s)$ and $\partial_1(v,s)=(C,\partial_1s)$ and extended to morphisms as follows. Given a cell
$$\xymatrix{
A \ar@{}[dr]|{\alpha} \ar[d]_{f} \ar[r]^{m}|-@{|} & B \ar[d]^{g} \\
C \ar[r]_n|-@{|} & D
}$$
take $\partial_0\alpha$ to be the morphism $f\colon (A,\partial_0s) \to (B,\partial_0t)$ and analogously for $\partial_1\alpha$.  These are well-defined by the commutativity conditions coming with the morphism of spans $M\alpha$.  The assignments are then functorial by that assumed for $M$.
\end{construction}

\begin{construction}[Actions] \label{actions on module} Form the pullback of $\partial_0\colon \dblelt(M)\to\dblelt(F)_0$ along the target projection $\dblelt(F)_1\to\dblelt(F)_0$ and give assignments
\[ L\colon \dblelt(F)_1\times_{\dblelt(F)_0}\dblelt(M)\to\dblelt(M)
\]
that will amount to an action. Summarize these assignments on objects and arrows at once by the picture:
$$\xymatrix{
(A,x)\ar@{}[dr]|{\alpha} \ar[d]_f \ar[r]^{(m,u)}|-@{|} & (B, y)\ar[d]^g & (p,r)\ar[d]^{\beta} \ar@{}[drr]|{\mapsto} & & (p\otimes m, \lambda(u,r)) \ar[d]^{\beta\otimes\alpha} \\
(C,z) \ar[r]_{(n,v)}|-@{|} & (D,w) & (q,s) & & (q\otimes n,\lambda(v,s))
}$$
where $\lambda$ is the action cell coming with $M$. Of course $\alpha$ and $\beta$ are composable by the construction of the pullback, but it needs to be seen that the composite $\beta\otimes\alpha$ does give a morphism of $\dblelt(M)$  But this is equivalent to the validity of the equation
\[ M(\beta\otimes\alpha)(\lambda(v,s))= \lambda(u,r)
\]
But this holds by the naturality condition for $\lambda$ in Definition \ref{def: module}, since $u= F\alpha(v)$ and $M\beta(s)=r$ both hold by the construction of morphisms in $\dblelt(M)$.  So defined, $L$ is a functor by the strict interchange law in $\mathbb B$; by the fact that $M$ is strictly horizontally functorial; and by the normalization hypothesis for units.  A functor $R$ for a right action of $\dblelt(G)_1$ on $\dblelt(M)$ is constructed analogously.  It remains to see that the action axioms are satisfied and that they are suitably compatible, yielding an internal profunctor.
\end{construction}

\begin{prop} \label{elts welldefn on modules} The assignments of Construction \ref{actions on module} are well-defined functors yielding an internal profunctor between discrete double fibrations $\dblelt(M) \colon\dblelt(F)\slashedrightarrow \dblelt(G)$.
\end{prop}
\begin{proof}  The action functors $L$ and $R$ are unital by the normalization assumption for vertical composition with units in $\mathbb B$.  Required are action iso cells such as 
$$\xymatrix{
\dblelt(F)_1\times_{\dblelt(F)_0}\dblelt(F)_1\times_{\dblelt(F)_0}\dblelt(M)\ar[d]_{\otimes\times 1}\ar[rr]^{\qquad1\times L} && \dblelt(F)_1\times_{\dblelt(F)_0}\dblelt(M) \ar[d]^{L}\\
\dblelt(F)_1\times_{\dblelt(F)_0}\dblelt(M)\ar[rr]_L
&& \dblelt(M)
}$$
and similarly for $R$. But chasing an object of the domain around either of the square as above and comparing, commutativity is given by associativity of proarrow composition in $\mathbb B$. Lastly, the actions $L$ and $R$ should be compatible in the sense that
$$\xymatrix{
\dblelt(F)_1\times_{\dblelt(F)_0}\dblelt(M)\times_{\dblelt(G)_0}\dblelt(G)_1 \ar[d]_{L\times 1}\ar[rr]^{\qquad 1\times R} & &
\dblelt(F)_1\times_{\dblelt(F)_0}\dblelt(M) \ar[d]^L \\
\dblelt(M)\times_{\dblelt(G)_0}\dblelt(G)_1 \ar[rr]_R & &
\dblelt(M)
}$$
commutes. But again chasing objects and arrows around each side of the square shows that commutativity follows from the compatibility assumption in Definition \ref{def: module}. \end{proof}

\begin{construction}[Elements from a Multimodulation]  Start with a multimodulation of contravariant lax $\spn$-valued functors
$$\xymatrix{
F^0 \ar@{}[drrr]|{\mu} \ar[d]_\tau \ar[r]^{M_1}|-@{|} & F^1 \ar[r]^{M_2}|-@{|} & \cdots \ar[r]^{M_k}|-@{|} & F^k \ar[d]^{\sigma} \\
G^0 \ar[rrr]_N|-@{|} & & & G^1
}$$
as in Definition \ref{multimodulation defn}.  This means that there are projection spans $\dblelt(F^{i-1})_0\leftarrow \mathscr \dblelt(M_i) \to\dblelt(F^{i})$ and one for $\mathscr N$, each with appropriate left and right actions as in Construction \ref{module to elements construction}. Define what will be a functor 
\[ \dblelt(\mu)\colon\dblelt(M_1) \times_{\dblelt(F^1)_0}\cdots \times_{\dblelt(F^{k-1})_0}\dblelt(M_k)\longrightarrow \dblelt(N)
\]
in the following way.  On objects take
\[ ((m_1,s_1),\dots, (m_k,s_k))\mapsto ([\mathbf m],\mu_{\mathbf m}(s))
\]
where $\mu_{\mathbf m}$ is the given function coming with $\mu$ and $s=(s_1,\dots s_k)$. An arrow of the supposed source is a sequence of externally composable cells $\theta_i\colon (m_i,s_i) \to (n_i,t_i)$.  Assign to such a sequence the morphism of $\mathscr N$ represented by their composite
\[ (\theta_1,\dots, \theta_k)\mapsto \theta_k\otimes\theta_{k-1}\otimes\cdots\otimes \theta_1.
\]
This does define a morphism $([\mathbf m],\mu_{\mathbf m}(s))\to ([\mathbf n],\mu_{\mathbf n}(t))$ of $\dblelt(N)$ by the strict composition for the module $N$ as in Definition \ref{def: module}. This functor has several naturality and equivariance properties, coming from the assumed properties of the original multimodulation $\mu$.  For example, notice that $\dblelt(\mu)$ commutes with the projections and the $0$-level of the induced double functors $\dblelt(\tau)_0$ and $\dblelt(\sigma)_0$ by construction.  Further properties are summarized in the next result.
\end{construction}

\begin{prop} \label{elts welldefn on cells}  The functor $\dblelt(\mu)$ of Construction \ref{module to elements construction} defines a multicell between internal profunctors of the form
$$\xymatrix{
\dblelt(F^0) \ar@{}[drrr]|{\dblelt(\mu)} \ar[d] \ar[r]^{\dblelt(M_1)}|-@{|} &\dblelt(F^1)\ar[r]^{\;\;\;\dblelt(M_2)}|-@{|} &\cdots \ar[r]^{\dblelt(M_k)\;\;\;\;}|-@{|}& \dblelt(F^k)\ar[d] \\
\dblelt(G^0) \ar[rrr]_{\dblelt(N)}|-@{|} & & & \dblelt(G^1)
}$$
This completes assignments for the elements functor $\dblelt(-)\colon \dbllax(\mathbb B^{op},\spn)\to \dblfib(\mathbb B)$ between virtual double categories.
\end{prop}
\begin{proof}  The appropriate commutativity at the $0$-level was observed above.  That the action is left equivariant is the statement that the square
$$\xymatrix{
\dblelt(F^0)_1\times_{\dblelt(F^0)_0}\dblelt(M_1)\times_{\dblelt(F^1)_0}\cdots \times_{\dblelt(F^{k-1})_0}\dblelt(M_k) \ar[rrr]^{\qquad\qquad\qquad \dblelt(\tau)\times\dblelt(\mu)} \ar[d]_{L\times 1} & & &
\dblelt(G^0)\times_{\dblelt(G^0)_0}\dblelt(N) \ar[d]^{L}\\
\dblelt(M_1)\times_{\dblelt(F^1)_0}\cdots \times_{\dblelt(F^{k-1})_0}\dblelt(M_k) \ar[rrr]_{\qquad\dblelt(\mu)} & & & 
\dblelt(N)
}$$
commutes. But chasing an object of the upper left corner around both sides of the square reveals that commutativity at the object level is precisely the left equivariance condition in Definition \ref{multimodulation defn}. Right and inner equivariance follow by the same type of argument. The assignments are already known to be well-defined on objects and morphisms. It follows easily that these assignments are suitably functorial in the sense of virtual double categories. \end{proof}

\subsection{Extending the Pseudo-Inverse}

Likewise, the pseudo-inverse of Lemma \ref{lemma:psd inv is a functor} extends to a functor of virtual double categories.

\begin{construction}[Pseudo-Inverse for Modules] \label{pseudo-inverse for modules construction} Let $M\colon P\slashedrightarrow Q$ denote an internal profunctor between discrete double fibrations $P\colon \mathbb E\to\mathbb B$ and $Q\colon \mathbb G\to\mathbb B$ as in Remark \ref{description of internal profunctors and multicells}. Construct what will be a module $F_M\colon F_P\slashedrightarrow F_Q$ between the associated lax functors $F_P$ and $F_Q$ from Lemma \ref{pseudo-inverse} in the following way.  Let $\mathscr M_m$ denote the inverse image of the proarrow $m\colon B\slashedrightarrow C$ of $\mathbb B$ under the discrete fibration $\Pi\colon \mathscr M\to\mathbb B_1$ coming with $M$.  To each such proarrow $m$ assign the span of sets
\[ \mathbb E_B \xleftarrow{\partial_0} \mathscr M_m \xrightarrow{\partial_1} \mathbb G_C
\]
Note that this is well-defined by the first assumed commutativity condition for $M$.  To each cell of $\mathbb B$, assign the morphism of spans
$$\xymatrix{
A\ar@{}[dr]|{\theta}\ar[d]_f \ar[r]^{m}|-@{|} & B \ar[d]^g & \ar@{}[drr]|{\mapsto} & & & \mathbb E_B \ar[d]_{f^*} & \ar[l] \mathscr M_n \ar[d]^{\theta^*} \ar[r] & \mathbb G_C \ar[d]^{g^*} \\
C \ar[r]_n|-@{|} & D & & & & \mathbb E_A & \ar[l] \mathscr M_m \ar[r] & \mathbb G_D
}$$
with $\theta^*\colon \mathscr M_n\to\mathscr M_m$ given by taking an object of $\mathscr M$ over $n$ to the domain of the unique morphism of $\mathscr M$ above $\theta$ via $\Pi\colon \mathscr M\to\mathbb B_1$.  This is well-defined and makes a span morphism. To complete the data, start with composable vertical arrows $m\colon A\slashedrightarrow B$ and $n\colon B\slashedrightarrow B$ and give assignments $\lambda$ and $\rho$ by using the given actions $L$ and $R$, taking
\[ \lambda_{m,n}\colon \mathbb E_m\times_{\mathbb E_D}\mathscr M_n \to \mathscr M_{n\otimes m} \qquad (\tilde m, \tilde n)\mapsto L(\tilde m, \tilde n)
\]
and similarly for $\rho$.  These are well-defined by the second row of commutativity conditions in Definition \ref{def: module}.  Additionally, these maps commute with the projections, in the sense that the diagrams
$$\xymatrix{
\mathbb E_A \ar@{=}[d] &\ar[l] \mathscr M_m\times_{\mathbb G_B}\mathbb Gn \ar[d]^{\rho} \ar[r] & \mathbb G_C \ar@{=}[d] & & & \mathbb E_A \ar@{=}[d] &\ar[l] \mathbb E_u\times_{\mathbb E_B}\mathscr M_n \ar[d]^\lambda \ar[r] & \mathbb G_C \ar@{=}[d] \\
\mathbb E_A & \ar[l] \mathscr M_{n\otimes m} \ar[r] & \mathbb G_C & & & \mathbb E_A & \ar[l] \mathscr M_{n\otimes m} \ar[r] & \mathbb G_C
}$$
both commute. 
\end{construction}

\begin{prop}  The assignments of Construction \ref{pseudo-inverse for modules construction} make $F_M\colon F_P\slashedrightarrow F_Q$ a module in the sense of Definition \ref{def: module}.
\end{prop}
\begin{proof}  All of the requirements for $F_M$ to be a module between lax functors are met by the corresponding properties of the original module $M$, together with the fact that $\Pi\colon \mathscr M\to\mathbb B_1$ is a discrete fibration. \end{proof}

\begin{construction}[Pseudo-Inverse Assignment on Modulations] \label{ps inv multimods} Start with a modulation $U$ in $\dblfib(\mathbb B)$ as in Remark \ref{description of internal profunctors and multicells}.  Thus, in particular, we have a functor
\[ U\colon \mathscr M^1\times_{\mathbb E^1_0} \cdots \times_{\mathbb E^{k-1}_0} \mathscr M^k\longrightarrow \mathscr N
\]
commuting with the projections to the end factors and commuting with the $(k-1)$-fold proarrow composition on $\mathbb B$.  Required is a multi-modulation
$$\xymatrix{
F_{P^0} \ar@{}[drrr]|{F_U} \ar[r]^{F_{M_1}}|-@{|} \ar[d]_{F_H} & F_{P^1} \ar[r]^{F_{M_1}}|-@{|} & \cdots \ar[r]^{F_{M_k}}|-@{|} & F_{P^k} \ar[d]^{F_K} \\
F_{Q^0} \ar[rrr]_{F_N}|-@{|} & & & F_{Q^1} 
}$$
Unpacking the constructions at a path of proarrows $\mathbf m=(m_1,\dots m_k)$, this is just to ask for a corresponding set function 
\[ \mathscr M^1_{m_1}\times_{\mathbb E^1_{A_1}}\times\cdots \times_{\mathbb E^{k-1}_{A_{k-1}}}\mathscr M_{m_k}\longrightarrow \mathscr N_{[\mathbf m]}
\]
which is given by the arrow part of $U$, namely, $U_1$. Well-definition follows from the fact that $U$ commutes with the projections to $\mathbb B$ and the $(k-1)$-fold iterated proarrow composition of $\mathbb B$.
\end{construction}

\begin{prop}  The choice of $(F_U)_{\mathbf m}= U_1$ in Construction \ref{ps inv multimods} results in a multimodulation of modules between lax functors as in Definition \ref{multimodulation defn}. This extends the functor $F_{(-)}$ to a functor of virtual double categories
\[ F_{(-)} \colon\dblfib(\mathbb B) \longrightarrow \dbllax(\mathbb B^{op},\spn).
\]
\end{prop}
\begin{proof}  The horizontal naturality condition holds by construction of the transition functions corresponding to cells $\theta$ and because $\Pi_1\colon \mathscr N\to\mathbb B_1$ is a discrete fibration.  Right, left and inner equivariance then follow from the corresponding properties of the original functor $U$.  \end{proof}

\subsection{The Equivalence of Virtual Double Categories}

The extended elements construction and the purported pseudo-inverse induce an equivalence of virtual double categories, leading to the full representation theorem, namely, Theorem \ref{Main Theorem 2} below.

\begin{construction} \label{equiv 1 redux} Extend the isomorphism of Construction \ref{equiv iso 1} to an isomorphism of functors of virtual double categories. The required multimodulation for the cell-components is straightforward to produce. Given $M\colon H\slashedrightarrow G$ and a proarrow $m\colon B\slashedrightarrow C$ of $\mathbb B$, define a function
\[ \eta_{M,m}\colon Mm\longrightarrow \dblelt(M)_m\qquad  s\mapsto (m,s)
\]
again just adding in an index. This is a bijection fitting into a morphism of spans
$$\xymatrix{
HB \ar[d]_{\eta_{H,B}} & \ar[l] Mm \ar[d]^{\eta_{M,m}} \ar[r] & GC \ar[d]^{\eta_{G,C}} \\
\dblelt(H)_B & \ar[l] \dblelt(M)_m \ar[r] & \dblelt(G)_C
}$$
that defines the required invertible modulation
$$\xymatrix{
\ar@{}[dr]|{\eta_M} H \ar[d]_{\eta_H} \ar[r]^{M}|-@{|} & G \ar[d]^{\eta_G} \\
\dblelt(F_H) \ar[r]_{\dblelt(M)}|-@{|} & \dblelt(F_G)
}$$
as in Definition \ref{multimodulation defn} by construction of the elements functor and its purported pseudo-inverse.  This is easy to check from the definitions.
\end{construction}  

\begin{prop} The assignments in Construction \ref{equiv 1 redux} yield a natural isomorphism of functors of virtual double categories $\eta\colon 1\cong \dblelt(F_{(-)})$.
\end{prop}
\begin{proof}  As discussed above, the components of the purported transformation are all well-defined, so it remains only to check the ``cell naturality" condition of Definition \ref{def: transf of funct of virt dble cats}.  Start with a generic multimodulation
$$\xymatrix{
\ar@{}[drrr]|{\mu} H^0 \ar[d]_\tau \ar[r]^{M_1}|-@{|} & H^1 \ar[r]^{M_2}|-@{|} & \cdots \ar[r]^{M_k} & H^k \ar[d]^\sigma \\
G^0 \ar[rrr]_{N}|-@{|} & & & G^1 
}$$
in $\dbllax(\mathbb B^{op},\spn)$.  By construction, the composite on right side of the condition sends an $k$-tuple $s=(s_1,\dots, s_k)$ to $((m_1,s_1), \dots, (m_k,s_k))$ and then to $([\mathbf m],\mu_{[\mathbf m]}(s))$; where as that on the left side sends the same element to $\mu_{[\mathbf m]}(s)$ first and then to $([\mathbf m],\mu_{[\mathbf m]}(s))$.  The point is that by construction evaluating and indexing commute.  In any case, the two sides are equal, and $\eta$ so defined is a natural isomorphism as claimed.  \end{proof}

\begin{construction} \label{equiv 2 redux} Extend the natural isomorphism of Construction \ref{equiv iso 2} to one of functors of virtual double categories $\epsilon \colon \dblelt(F_{(-)})\cong 1$. Take an internal profunctor $M\colon P \slashedrightarrow Q$ of discrete double fibrations $P\colon\mathbb E\to\mathbb B$ and $Q\colon\mathbb G\to\mathbb B$. The required span morphism
$$\xymatrix{
\mathbb E_0 \ar@{=}[d] & \ar[l] \dblelt(F_M) \ar@{-->}[d]^{\epsilon_M} \ar[r] & \mathbb G_0 \ar@{=}[d] & & (m,s) \xrightarrow{\alpha} (n,t) \ar@{}[d]|{\rotatebox[origin=c]{270}{$\mapsto$}} \\
\mathbb E_0 & \ar[l] \mathscr M \ar[r] & \mathbb G_0 & & s\xrightarrow{!} t
}$$
is given by the fact that $M\colon \mathscr M \to\mathbb B_1$ is a discrete fibration. It is a functor and an isomorphism by uniqueness and equivariant by construction, making a morphism of internal profunctors.
\end{construction}

\begin{prop} The assignments in Construction \ref{equiv 2 redux} are a natural isomorphism of functors of virtual double categories.
\end{prop}
\begin{proof} The one condition to check is the ``Cell Naturality" of Definition \ref{def: transf of funct of virt dble cats}. Taking a multicell between internal profunctors
$$\xymatrix{
\ar@{}[drrr]|{\mu} P^0 \ar[d]_H \ar[r]^{M_1}|-@{|} & P^1 \ar[r]^{M_2}|-@{|} & \cdots \ar[r]^{M_k} & P^k \ar[d]^K \\
Q^0 \ar[rrr]_{N}|-@{|} & & & Q^1 
}$$
the statement of the condition reduces to checking that the equation
\[ \mu\circ (\epsilon_{M_1}\times \cdots \epsilon_{M_k}) = \epsilon_N\circ \dblelt(F_\mu)
\]
holds. But this is true by definition of $\dblelt(F_\mu)$ and the components of $\epsilon$. For on the one hand, a $k$-tuple $((m_1,s_1),\dots, (m_k,s_k))$ is sent to $s= (s_1,\dots, s_k)$ and then to $\mu(s)$. On the other hand, the same $k$-tuple is sent to $([\mathbf m],s)$ by $\dblelt(F_\mu)$ and then to $\mu(s)$ by $\epsilon_N$. The same kind of check works at the level of arrows. Thus, $\epsilon$ is a natural isomorphism.\end{proof}

\begin{theo} \label{Main Theorem 2}  There is an equivalence of virtual double categories
\[  \dbllax(\mathbb B^{op},\spn)\simeq \dblfib(\mathbf B)
\]
for any double category $\mathbb B$ induced by the elements functor $\dblelt(-)$.
\end{theo}
\begin{proof} This is proved by Propositions \ref{equiv 1 redux} and \ref{equiv 2 redux}. \end{proof}

\section{Prospectus}

Let us close with a preview of forthcoming work relating to the present results.

\subsection{Monadicity}

It is well-known that ordinary discrete fibrations over a fixed base are monadic over a slice of the category of sets. This fact is of central importance in the elementary axiomatization of results relating to presheaves and sheaves in the language of an elementary topos \cite{DiaconescuThesis}, \cite{DiaconescuChangeOfBase}. In this development, ``base-valued functors" (i.e. presheaves) are axiomatized as certain algebras for a monad on a slice of the ambient topos.

Any parallel development in a double categorical setting of these presheaf results will require an analogous monadicity result. Forthcoming work will establish that discrete double fibrations over a fixed base double category are monadic over a certain slice of the double category of categories. Pursing some notion of ``double topos" as a forum for formal category theory, this will give a setting for elementary axiomatization of elements of presheaf and Yoneda theory for double categories.

\subsection{Double Fibrations}

The main definition of the paper anticipates the natural question about whether there is a more general notion of a ``double fibration" of which a discrete double fibration is a special case. For recall that each ordinary discrete fibration $F\colon\F\to\C$ between 1-categories is a (split) fibration in a more general sense. Split fibrations of course have lifting properties with respect to certain compatibly chosen ``cartesian arrows" and correspond via a category of elements construction to contravariant category-valued 2-functors on the base category. The question of the double-categorical analogue is the subject of forthcoming work with G. Cruttwell, D. Pronk and M. Szyld. The evidence of the correctness of the proposed definition will be a representation theorem like Theorem \ref{Main Theorem 1} in the present paper, but suitably upping the dimension of the representing structure.

\section{Acknowledgments}

Thanks to Dr. Dorette Pronk for supervising the author's thesis where some of the ideas for this paper were first conceived. Thanks also to Geoff Cruttwell, Dorette Pronk and Martin Szyld for a number of helpful conversations, comments, and suggestions on the material in this project and related research. Thanks in particular to Geoff Cruttwell for clarifying some questions on virtual double categories. Special thanks are due to Bob Par\'e for his encouragement when the author was just getting into double-categorical Yoneda theory.

\end{document}